\RequirePackage{fix-cm}
\documentclass[smallextended]{svjour3}       
\smartqed  
\usepackage{csquotes}
\usepackage{graphicx}
\usepackage{amsfonts}
\usepackage{amsmath}
\usepackage{amssymb}
\usepackage{geometry}
\usepackage{color}
\usepackage[hidelinks]{hyperref}

\usepackage{placeins}

\newcommand{\B}[1]{\boldsymbol{#1}}

\usepackage{algorithm}      
\usepackage{algpseudocode}

\usepackage{cite}

\begin{document}

\title{A General Probability Density Framework for Local Histopolation and Weighted Function Reconstruction from Mesh Line Integrals
}

\titlerunning{Weighted Function Reconstruction from Mesh Line Integrals}        

\author{Francesco Dell'Accio \and 
        Allal Guessab \and 
        Mohammed Kbiri Alaoui
        \and
        Federico Nudo 
}


\institute{  Francesco Dell'Accio \at
             Department of Mathematics and Computer Science, University of Calabria, Rende (CS), Italy\\
\email{francesco.dellaccio@unical.it} 
         \and 
             Mohammed
Kbiri \at
             Department of Mathematics, College of Science, King Khalid
University, P.O. Box 9004, 61413 Abha, Saudi Arabia. \\
              \email{mka\_la@yahoo.fr}
                   \and 
           Allal Guessab \at
             Avenue Al Walae, 73000, 
Dakhla, Morocco \\
    \email{guessaballal7@gmail.com}
               \and 
            Federico Nudo (corresponding author) \at
              Department of Mathematics and Computer Science, University of Calabria, Rende (CS), Italy \\
    \email{federico.nudo@unical.it}
}

\date{Version: September 03, 2025}

\maketitle

\begin{abstract}
In this paper, we study the reconstruction of a bivariate function from weighted integrals along the edges of a triangular mesh, a problem of central importance in tomography, computer vision, and numerical approximation. Our approach relies on local histopolation methods defined through unisolvent triples, where the edge weights are induced by suitable probability densities. In particular, we introduce two new two-parameter families of generalized truncated normal distributions, which extend classical exponential-type laws and provide additional flexibility in capturing local features of the target function. These distributions give rise to new quadratic reconstruction operators that generalize the standard linear histopolation scheme, while retaining its simplicity and locality. We establish their theoretical foundations, proving unisolvency and deriving explicit basis functions, and we demonstrate their improved accuracy through extensive numerical tests. Moreover, we design an algorithm for the optimal selection of the distribution parameters, ensuring robustness and adaptivity of the reconstruction. Finally, we show that the proposed framework naturally extends to any bivariate function whose restriction to the edges defines a valid probability density, thus highlighting its generality and broad applicability.
\end{abstract}


\section{Introduction}

The classical concept of polynomial interpolation can be extended beyond pointwise evaluations 
to situations where the available information consists of general linear functionals~\cite{Rivlin}. 
When these data take the form of integrals over geometric domains, the problem is referred to as 
\emph{histopolation}~\cite{Robidoux}, and the corresponding techniques are known as 
\emph{histopolation methods}.  

Histopolation provides a general framework for reconstructing functions from integral data, 
a problem of central importance in many areas of science and engineering. 
In practical applications, direct pointwise samples of the target function are often 
unavailable; instead, only integral measurements are accessible. 
This situation arises prominently in \emph{computed tomography} and \emph{medical imaging}, 
where measurements typically consist of line integrals~\cite{kak2001principles,natterer2001mathematics,palamodov2016reconstruction}. 
Beyond tomography, histopolation has found applications in \emph{computer vision}, 
\emph{signal and image processing}~\cite{Bosner:2020:AOC}, \emph{curve and spline approximation}~\cite{Fischer:2005:MPR,Fischer:2007:CSP,Siewer:2008:HIS,Hallik:2017:QLR}, 
\emph{fractal functions}~\cite{Barnsley:2023:HFF}, and in the \emph{conservation of physical quantities} 
in numerical simulations~\cite{HiemstraJCP}. 
Recently, histopolation methods have received increasing attention~\cite{BruniErbFekete,bruni2024polynomial}, 
including global polynomial histopolation–regression in one~\cite{bruni2025polynomial} and several 
dimensions~\cite{bruno2025bivariate}, as well as weighted \emph{local} polynomial histopolation 
schemes in the bivariate setting~\cite{nudo1, nudo2, dell2025reconstructing,dell2025truncated,milov1}.

A standard strategy for these reconstruction problems is to partition the computational domain into 
subdomains (a mesh), approximate locally on each element, and assemble a global approximation from 
the local contributions. Such schemes are typically \emph{conforming} or \emph{nonconforming}, 
depending on whether the assembled approximation is continuous across element interfaces.

A local approximation method can be rigorously defined through the triple~\cite{Guessab:2022:SAB}
\[
\mathcal{M}_d=\left(S_d, \mathbb{F}_{S_d}, \Sigma_{S_d}\right),
\]
where
\begin{itemize}
    \item $S_d$ is a polytope in $\mathbb{R}^d$, $d \ge 1$;
    \item $\mathbb{F}_{S_d}$ is an $n$-dimensional space of trial functions on $S_d$;
    \item $\Sigma_{S_d}=\left\{\mathcal{L}_j: j=1,\dots,n\right\}$ is a set of linearly independent linear 
    functionals, called the \emph{degrees of freedom}.
\end{itemize}
The triple $\mathcal{M}_d$ is said to be \emph{unisolvent} if, whenever $f \in \mathbb{F}_{S_d}$ 
satisfies
\begin{equation*}
\mathcal{L}_j(f)=0, \quad j=1,\dots,n,    
\end{equation*}
then necessarily $f=0$. A set of functions
\begin{equation*}
    B=\left\{\varphi_i\,:\, i=1,\dots,n\right\}
\end{equation*}
such that
\begin{equation*}
    \operatorname{span}\left\{\varphi_i\,:\, i=1,\dots,n\right\}=S_d
\end{equation*}
and
\begin{equation*}
    \mathcal{L}_j\left(\varphi_i\right)=\delta_{ij}
\end{equation*}
is said to form the basis functions associated with the triple $\mathcal{M}_d$. If the degrees of freedom are integral functionals, the scheme is also referred to as a 
\emph{local histopolation method}.

The simplest example is the linear nonconforming histopolation method on triangles, 
where edge integrals act as degrees of freedom and piecewise linear functions form the 
trial space. While attractive for its simplicity and efficiency, this approach has 
important limitations: its piecewise linear structure fails to reproduce oscillatory 
behaviors or sharp gradients, and higher accuracy can only be achieved through 
substantial mesh refinement.  

In this work we propose an enrichment strategy that enhances local histopolation 
schemes by introducing quadratic trial spaces together with weighted degrees of freedom. 
The key idea is to employ probability densities on the edges of the mesh, which act 
as flexible weights and introduce tunable parameters. 
In particular, we construct two new two-parameter families of generalized truncated 
normal distributions, which extend classical exponential-type laws and include 
special or limiting cases (such as truncated normal and beta-type densities). 
These families naturally lead to enriched quadratic operators that significantly 
improve reconstruction accuracy while retaining the locality and simplicity of 
the original method.  

From a theoretical point of view, we prove unisolvency of the enriched triples 
and derive explicit closed-form basis functions. 
From a practical perspective, we validate the proposed operators through 
extensive numerical experiments, which demonstrate their effectiveness in 
reconstructing both smooth and non-smooth functions. 
Moreover, as shown in Section~\ref{sec3}, the framework is not limited 
to the two families introduced here: the same reasoning applies to any 
bivariate function whose restriction to the edges defines a valid 
probability density, thus highlighting the generality and robustness of 
the proposed approach.  

The paper is organized as follows.
In Section~\ref{sec1}, we introduce two families of generalized truncated normal distributions, analyze their main properties, and define the corresponding families of quadratic histopolation operators. Moreover, we describe an algorithm for the optimal selection of the distribution parameters, which further enhances the robustness and adaptivity of the reconstruction. In Section~\ref{sec3}, we develop the general framework for enriched quadratic histopolation operators and establish their theoretical foundations. In Section~\ref{sec4}, we present numerical experiments that confirm the accuracy and effectiveness of the proposed methods. Finally, in Section~\ref{sec5}, we conclude with a summary of the contributions and discuss possible directions for future research, including three-dimensional extensions and applications to imaging and approximation problems.

\section{ Two-parameter families of generalized truncated normal distributions}\label{sec1}
\subsection{Problem statement}
The problem we want to address is the following: given only integral measurements of an 
unknown bivariate function along the edges of a triangular mesh, reconstruct an accurate 
approximation of the function inside each triangle. 
In the classical local histopolation method, these measurements are taken as unweighted 
edge integrals and the reconstruction space is piecewise linear. While simple and efficient, 
this approach lacks the ability to capture curved or oscillatory features unless the mesh 
is strongly refined. 

\noindent
Our objective is to overcome this limitation by enriching the standard framework. 
We introduce probability density functions as edge weights in the integral data, so that 
each measurement carries additional information about the local behavior of the function. 
This enrichment is feasible under the assumption that additional integral data are available, 
beyond the classical unweighted edge averages. In this setting, the reconstruction naturally 
leads to quadratic operators that preserve locality and simplicity while significantly improving 
approximation accuracy. The role of this section is to present two flexible two-parameter 
distribution families that serve as the building blocks of the proposed weighted histopolation scheme.

\subsection{Definition of the distribution families}
In this section, we introduce two distinct two-parameter families of univariate distributions, each generated by suitable bivariate weight functions defined on a triangle of a given mesh. These families include, as limiting or special cases, several well-known probability laws. They provide the building blocks for a broad class of new piecewise quadratic reconstruction operators, which in turn form the basis of a nonconforming histopolation method locally determined by a unisolvent triple. The goal is to reconstruct a function using information derived from weighted integrals over the edges of a triangular mesh. To this end, we formulate the probability density functions (hereafter, pdfs) associated with the edges of a nondegenerate triangle $T$ with vertices $\B{v}_1, \B{v}_2, \B{v}_3$ and barycentric coordinates $\lambda_1,\lambda_2,\lambda_3$. These coordinates are the unique affine functions on $T$, namely
\begin{equation}\label{prop2}
  \lambda_i\left(t\B{x} + (1-t)\B{y}\right) 
  = t\lambda_i(\B{x}) + (1-t)\lambda_i(\B{y}), 
  \quad \B{x},\B{y}\in T,\quad t\in[0,1],
\end{equation}
which satisfy the partition of unity property,
\begin{equation*}
  \sum_{i=1}^3 \lambda_i(\B{x}) = 1, 
  \quad \B{x}\in T,
\end{equation*}
and the \emph{Kronecker delta property} at the vertices,
\begin{equation*}
  \lambda_i\left(\B{v}_j\right) = \delta_{ij} =
  \begin{cases}
    1, & i=j,\\[4pt]
    0, & i\neq j,
  \end{cases}
  \quad i,j=1,2,3.
\end{equation*}
As a consequence, we have
\begin{equation}\label{propstar}
  \lambda_i(\B{x}) = 0, 
  \quad \B{x}\in s_i, 
  \quad i=1,2,3.
\end{equation}
For the forthcoming analysis, we recall the \emph{lower incomplete gamma function}, which will play a central role. It is defined for $s>0$ and $z> 0$ by
\begin{equation}\label{igamma}
  \gamma(s,z) = \int_{0}^{z} t^{s-1} e^{-t}dt.
\end{equation}
Among its numerous properties, one that we shall exploit is the limiting relation
\begin{equation}\label{limgamma}
  \lim_{z\to 0} \frac{\gamma(s,z)}{z^s} = \frac{1}{s}.
\end{equation}
Indeed, since $$e^{-z} \leq e^{-t} \leq 1, \quad 0 \leq t \leq z,$$ we obtain
\begin{equation*}
  t^{s-1}e^{-z} \leq t^{s-1}e^{-t} \leq t^{s-1}.
\end{equation*}
Integrating over $[0,z]$ with respect to $t$ and dividing by $z^s$, we obtain the inequalities
\begin{equation*}
  \frac{e^{-z}}{s}\leq \frac{\gamma(s,z)}{z^s} \leq \frac{1}{s}.
\end{equation*}
Taking the limit as $z\to 0$ establishes~\eqref{limgamma}.
This relation will be crucial for analyzing the convergence of our densities when one of their parameters tends to infinity. For convenience, we introduce the \emph{modified incomplete gamma function}
\begin{equation}\label{migamma}
  \gamma^{\mathrm{mod}}(s,z) := \frac{\gamma(s,z)}{z^s}=\frac{1}{z^s}\int_{0}^{z} t^{s-1} e^{-t}dt,
\end{equation}
so that~\eqref{limgamma} takes the simpler form
\begin{equation}\label{modlimgamma}
  \lim_{z\to 0} \gamma^{\mathrm{mod}}(s,z) = \frac{1}{s}.
\end{equation}
If a numerical routine is available to evaluate $\gamma(s,z)$, then $\gamma^{\mathrm{mod}}(s,z)$ can be computed directly by scaling. For example, in Matlab one may use 
\texttt{gammainc(z,s)\,gamma(s)}, while in Mathematica the corresponding function is \texttt{Gamma[s,0,z]}.

\subsection{First family of distributions}\label{ss1} 
The first family is defined through the following \emph{bivariate} function with shape parameter $\mu \geq 1$ and scale parameter $\sigma>0$ by
\begin{equation}\label{exp2}
  K_{\sigma,\mu}(\B{x}) 
  = a_{\sigma,\mu} \left(H^2(\B{x})\right)^{\mu -1}
  e^{-\frac{1}{2}\left(\frac{H(\B{x})}{\sigma^2}\right)^{\mu}},
  \quad \B{x}\in T,
\end{equation}
where $H$ is expressed in terms of barycentric coordinates as
\begin{equation}\label{gw}
  H(\B{x}) = 2\sum_{i=1}^{3}\lambda_i^2(\B{x}) - 1,
  \quad \B{x}\in T,
\end{equation}
and $a_{\sigma,\mu}$ is a normalization factor ensuring that integrals involving the associated generalized truncated normal distribution are properly scaled.  
It is worth noting that $K_{\sigma,\mu}$ is a weight function, 
as it is nonnegative throughout the triangle $T$. 
Furthermore, when restricted to any edge of $T$, 
$H$ reduces to a quadratic form, 
and the induced univariate functions not only serve as weights 
but also define valid probability density functions on the edges. An advantage of the general form~\eqref{exp2} is that it allows us to tune the free parameters $\sigma$ and $\mu$ when reconstructing a function with specific features. In particular, different behaviors can be produced by varying the exponential term in~\eqref{exp2}. \\

In the following, we parametrize the edge $s_j$ as
\begin{equation}\label{parm}
      \frac{1+t}{2}\B{v}_{j+1} + \frac{1-t}{2}\B{v}_{j+2}, 
      \quad t\in[-1,1].
\end{equation}
Then, by~\eqref{prop2} and~\eqref{propstar}, together with~\eqref{gw}, we obtain
\begin{equation}\label{hh}
  H\left( \frac{1+t}{2}\B{v}_{j+1} + \frac{1-t}{2}\B{v}_{j+2}\right) = t^2.
\end{equation}
We now define the first class of univariate pdfs associated with the weight function $K_{\sigma,\mu}$.  
It follows directly from the definition of $K_{\sigma,\mu}$ that on the edge $s_i$ of $T$ one has 
\begin{equation}
\label{gwn=10}
 k_{\sigma,\mu}(t):= K_{\sigma,\mu} \left(\frac{1+t}{2}\B{v}_{j+1} + \frac{1-t}{2}\B{v}_{j+2}\right) = a_{\sigma,\mu}  \left(t^2\right)^{2\mu-2} 
 e^{-\frac{1}{2} \left(\frac{t^2}{\sigma^2}\right)^{\mu} }, \quad  t\in \left[-1,1\right]. 
\end{equation}
An important observation is that $k_{\sigma,\mu}(t)$ is independent of the choice of edge of the triangle $T$.  
We shall henceforth work with the normalized probability density functions (pdfs) with unit zeroth moment, which we denote with a tilde. Thus, in order to obtain a normalized pdf, the coefficient $a_{\sigma,\mu}$ in the bivariate weight function $K_{\sigma,\mu}(\B{x})$ defined in~\eqref{exp2} is introduced so that
\begin{equation*}
    \int_{-1}^1 k_{\sigma,\mu}(t)dt = 1.
\end{equation*}
As shown in Lemma~\ref{pd2}, for any $\mu \geq 1$ and $\sigma > 0$, the normalization constant $a_{\sigma,\mu}$ is given by
\begin{equation*}
 a_{\sigma,\mu} = \frac{\mu }{ 2^{\frac{4\mu -3}{2\mu}} \sigma^{4\mu-3}    \gamma\left(\frac{4\mu-3}{2\mu},\frac{1}{2\sigma^{2\mu}}\right)},
\end{equation*}
which can be written more compactly in terms of the modified incomplete gamma function as
\begin{equation}\label{moddens1}
 a_{\sigma,\mu} = \frac{\mu }{    \gamma^{\mathrm{mod}}\left(\frac{4\mu-3}{2\mu},\frac{1}{2\sigma^{2\mu}}\right)}.
\end{equation}
Hence, for any $\sigma>0$ and $\mu\geq 1$, substituting~\eqref{moddens1} into~\eqref{gwn=10}, it follows that the normalized pdfs on $[-1,1]$ are given by
\begin{equation}\label{gwn=1}
  \widetilde{k}_{\sigma,\mu}(t)= \frac{\mu }{    \gamma^{\mathrm{mod}}\left(\frac{4\mu-3}{2\mu},\frac{1}{2\sigma^{2\mu}}\right)} 
   \left(t^2\right)^{2\mu-2}   e^{-\frac{1}{2} \left(\frac{t^2}{\sigma^2}\right)^{\mu}}, \quad  t\in \left[-1,1\right]. 
\end{equation}
As a first observation, this family extends the class of exponential-type distributions.  
Indeed, when $\mu = 1$, equation~\eqref{gwn=1} reduces to the well-known doubly truncated normal distribution on the interval $[-1,1]$, here written in terms of $\gamma^{\mathrm{mod}}$; see~\cite{Jawitz2004}. Indeed, by straightforward calculation, and using the notation~\eqref{modlimgamma}, this distribution can be written as
\begin{equation*}
   \widetilde{k}_{\sigma,1}(t)=  \frac{1}{\gamma^{\mathrm{mod}}\left(\frac{1}{2}, \frac{1}{2\sigma^{2}} \right)}  e^{-\frac{1}{2} \left(\frac{t^2}{\sigma^2}\right)
 }, \quad t\in \left[-1,1\right]. 
\end{equation*}
Recall that this latter distribution has no shape parameter, and thus cannot model all phenomena. Nevertheless, it has been extensively studied and plays a central role in probability theory, see, e.g.,~\cite{Xinyi2012functions} and the references therein. Therefore, our pdfs constitute a new two-parameter family of exponential type. In Figure~\ref{fig:densities}, we plot several instances of $\widetilde{k}_{\sigma,\mu}$, illustrating the strong sensitivity of the distribution to both the shape parameter $\mu$ and the scale parameter $\sigma$. This highlights the increased flexibility and importance of the generalized family.

    \begin{figure}
      \centering
\includegraphics[width=0.49\linewidth]{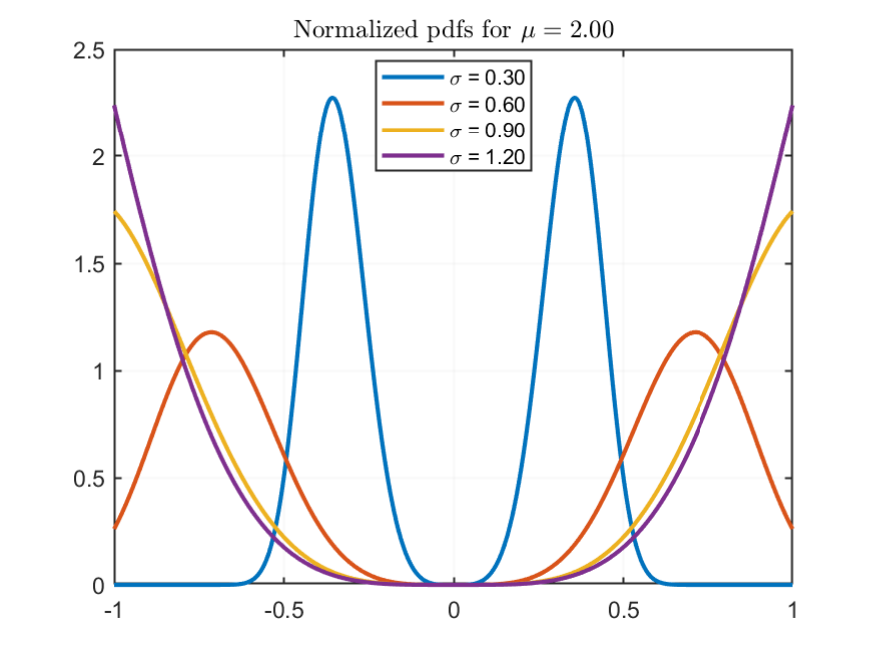}
\includegraphics[width=0.49\linewidth]{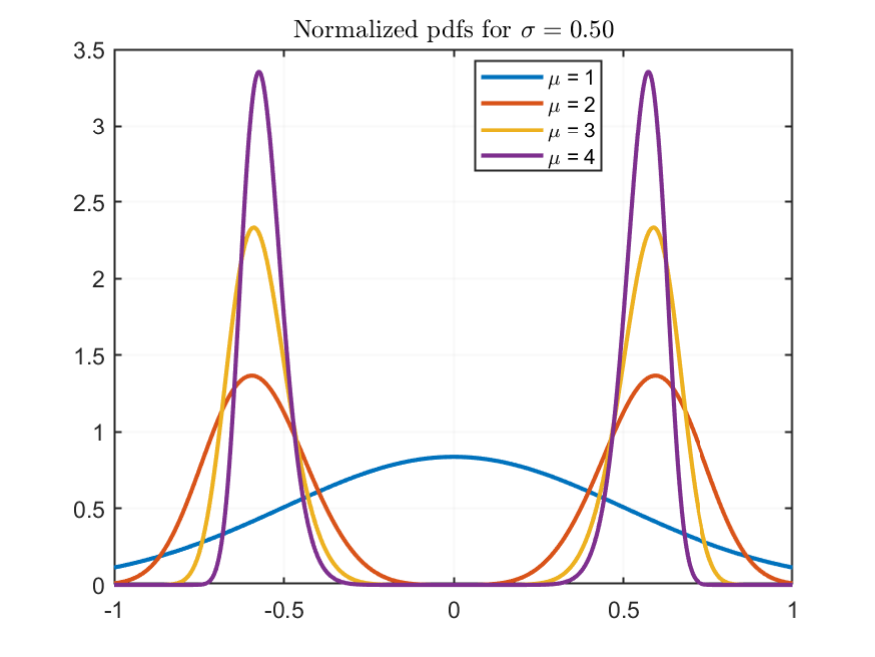}
      \caption{Normalized probability density functions $\widetilde{k}_{\sigma,\mu}$ on $[-1,1]$. 
Left: effect of the scale parameter, with fixed $\mu=2$ and different values of $\sigma$. 
Right: effect of the shape parameter, with fixed $\sigma=0.50$ and varying values of $\mu$, illustrating the increased sharpness and concentration around $t=0$. 
These plots emphasize the flexibility provided by the two parameters.}
      \label{fig:densities}
  \end{figure}

 \begin{remark} 
The pdfs $\widetilde{k}_{\sigma,\mu}$ also include a beta-type distribution as a limiting
case. Indeed, using~\eqref{modlimgamma} with 
\begin{equation*}
s=\frac{4\mu-3}{2\mu}, \quad z=\frac{1}{2\sigma^{2\mu}},     
\end{equation*}
 we obtain
\begin{equation}\label{prov1}
    \lim_{\sigma \to \infty} \gamma^{\mathrm{mod}}\left(\frac{4\mu-3}{2\mu},\frac{1}{2\sigma^{2\mu}}\right) 
    = \frac{2\mu}{4\mu-3}.
\end{equation}
Therefore, for any $t\in[-1,1]$, the limiting distribution is
\begin{equation}\label{w1inf}
  \widetilde{k}_{\infty,\mu}(t) 
  := \lim_{\sigma \to \infty} \widetilde{k}_{\sigma,\mu}(t) 
  = \frac{4\mu-3}{2}\left(t^2   \right)^{2\mu-2}.
\end{equation}
Thus, $\widetilde{k}_{\sigma,\mu}$ converges to a beta-type distribution as $\sigma \to \infty$.
 \end{remark}
In order to present our local reconstruction operator, which is based on the pdfs $\widetilde{k}_{\sigma,\mu}$, we now introduce some additional notations and recall the basic setting for clarity of exposition.  
The operator is built locally on each nondegenerate triangle $T$ of a given mesh.  
For notational convenience, we set
\begin{equation*}
    \B{v}_4 = \B{v}_1, \quad  \B{v}_5 = \B{v}_2,
\end{equation*}
as well as 
\begin{equation*}
   \lambda_4 = \lambda_1, \quad \lambda_5 = \lambda_2.  
\end{equation*}
Then, using the parametrization of the side $s_j$ given in~\eqref{parm}, we introduce the linear functionals 
\begin{equation}\label{lev} \mathcal{I}_j^{\mathrm{CH}}(f) =\frac{1}{\left\lvert s_j\right\rvert} \int_{s_j} f(\boldsymbol{x}) d\boldsymbol{x}=\frac{1}{2}\int_{-1}^1 f\left(\frac{1+t}{2}\B v_{j+1}+\frac{1-t}{2}\B v_{j+2}\right) d\boldsymbol{x}, \quad j=1,2,3, 
\end{equation}
and denote
 $$\Sigma^{\mathrm{CH}} = \left\{\mathcal{I}_j^{\mathrm{CH}} \, : \, j=1,2,3\right\}.$$
Thus, the linear nonconforming histopolation method can be locally represented by the triple
\begin{equation}\label{CRelement}
\mathcal{CH} = \left(T,  \mathbb{P}_1(T), \Sigma^{\mathrm{CH}}\right),    
\end{equation}
where $\mathbb{P}_1(T)$ is the space of linear polynomials on $T$, and its associated basis functions are expressed in barycentric coordinates as
\begin{equation*}
   \varphi_i^{\mathrm{CH}} = 1 - 2\lambda_i, 
   \quad i=1,2,3,
\end{equation*}
as can be directly verified. The reconstruction operator associated to the triple~\eqref{CRelement} is
\begin{equation}\label{pi1CH}
\begin{array}{rcl}
{\pi}^{{\mathrm{CH}}}: C(T) &\rightarrow& {\mathbb{P}}_{1}(T)
\\
f &\mapsto& \displaystyle{\sum_{j=1}^{3}}  \mathcal{I}^{\mathrm{CH}}_{j}\left(f\right)\varphi^{\mathrm{CH}}_{j}.
\end{array}
\end{equation}

To describe our first new two-parameter family of nonconforming histopolation operators, we introduce a class of enriched linear functionals, defined by
\begin{eqnarray}
  {\mathcal{I}}_{j}^{\mathrm{enr}}(f)&=& \int_{-1}^{1} 
 f\left(\frac{1+t}{2}\B{v}_{j+1} + \frac{1-t}{2}\B{v}_{j+2}\right) \widetilde{k}_{\sigma,\mu}(t)  dt, \quad j=1,2,3,\label{ex3qq} \\
 {\mathcal{L}}_{j}^{\mathrm{enr}}(f)&=& \int_{-1}^{1} {p}_{\sigma,\mu}(t)
 f\left(\frac{1+t}{2}\B{v}_{j+1} + \frac{1-t}{2}\B{v}_{j+2}\right) \widetilde{k}_{\sigma,\mu}(t)  dt, \quad j=1,2,3,\label{ex3qqq}
\end{eqnarray}
 where  
\begin{equation}\label{LegPol}
{p}_{\sigma,\mu}(t)=       t^2-    \frac{\gamma^{\mathrm{mod}}\left(\frac{4\mu-1}{2\mu}, \frac{1}{2\sigma^{2\mu}} \right)}{\gamma^{\mathrm{mod}}\left(\frac{4\mu-3}{2\mu}, \frac{1}{2\sigma^{2\mu}} \right)}.
\end{equation}
Since $\widetilde{k}_{\sigma,\mu}$ is a probability density, the functionals 
$\mathcal{I}_{j}^{\mathrm{enr}}$ and $\mathcal{L}_{j}^{\mathrm{enr}}$ can be interpreted as weighted averages of $f$ and $p_{\sigma,\mu}f$ along the side $s_j$ of $T$.  The rationale behind the choice of $p_{\sigma,\mu}$, which will become clear in Section~\ref{sec3}, lies in its central role in the subsequent analysis. 
Using~\eqref{hh} and the parametrization~\eqref{parm}, this polynomial can equivalently be expressed in terms of the function $H$ defined in~\eqref{gw} as
  \begin{eqnarray*}
     {p}_{\sigma,\mu}(t)&=&  \widetilde{H} \left(\frac{1+t}{2}\B{v}_{j+1} + \frac{1-t}{2}\B{v}_{j+2}\right) \\
     &:=& H \left(\frac{1+t}{2}\B{v}_{j+1} + \frac{1-t}{2}\B{v}_{j+2}\right) - \frac{\gamma^{\mathrm{mod}}\left(\frac{4\mu-1}{2\mu}, \frac{1}{2\sigma^{2\mu}} \right)}{\gamma^{\mathrm{mod}}\left(\frac{4\mu-3}{2\mu}, \frac{1}{2\sigma^{2\mu}} \right)}.
  \end{eqnarray*}
  \begin{remark}
      We remark that, using~\eqref{gwn=10}, the product $p_{\sigma,\mu}(t)\widetilde{k}_{\sigma,\mu}(t)$ appearing in the definition of 
$\mathcal{L}_{j}^{\mathrm{enr}}$ can be further written as
\begin{equation*}
   p_{\sigma,\mu}(t)\widetilde{k}_{\sigma,\mu}(t)
   = \left(\widetilde{H}K_{\sigma,\mu}\right)
     \left(\frac{1+t}{2}\B{v}_{j+1} + \frac{1-t}{2}\B{v}_{j+2}\right).
\end{equation*}
Moreover, we note that when $\mu=1$, equation~\eqref{w1inf} implies
$$\widetilde{k}_{\infty,1}(t) = \frac{1}{2},$$  
and in this particular case the enriched functionals in~\eqref{ex3qq} reduce to the classical unweighted histopolation functionals of~\eqref{lev}.

  \end{remark}

We define the associated set of enriched linear functionals as
 \begin{equation*}
\Sigma^{\mathrm{enr}}_{1,\sigma,\mu}(T)=  \left\{{\mathcal{I}}_{j}^{\mathrm{enr}}, {\mathcal{L}}^{\mathrm{enr}}_{j}\, : \, j=1,2,3\right\}
\end{equation*}
and introduce the enriched triple 
\begin{equation}\label{tripless}
\mathcal{H}^{\mathrm{enr}}_{1,\sigma,\mu}= \left(T,  \mathbb{P}_{2}(T),\Sigma^{\mathrm{enr}}_{1,\sigma,\mu}(T)\right),
\end{equation}
where  ${\mathbb{P}}_{2}(T)$ is the set of all quadratic polynomials  defined on $T.$ 
This triple can be regarded as a generalization of the linear (unweighted) histopolation scheme.  
We employ the proposed enriched quadratic operators $\mathcal{H}^{\mathrm{enr}}_{1,\sigma,\mu}$ as building blocks for the function reconstruction method.  
Their quadratic nature allows them to represent curved features and ensures greater accuracy compared to the linear case.

The first step in our analysis is to show that, for any $\sigma>0$ and $\mu\geq 1$, the system $\mathcal{H}^{\mathrm{enr}}_{1,\sigma,\mu}$ defines a well-posed unisolvent triple for $\mathbb{P}_2(T)$.  
The second step is to determine the associated basis functions.  
These functions are then employed to define the first two-parameter family of \emph{local function reconstruction} operators.  \\

We now summarize our main findings together with some immediate consequences.  
Detailed proofs will be provided at the end of this subsection.

\begin{theorem}\label{th1nnewfin11} 
 The triple $\mathcal{H}^{\mathrm{enr}}_{1,\sigma,\mu}$ is unisolvent for any $\sigma>0$ and $\mu\geq 1$.
\end{theorem}
We next derive closed-form expressions for the basis functions associated with the enriched triple $\mathcal{H}^{\mathrm{enr}}_{1,\sigma,\mu}$.  
Specifically, we determine a basis
\begin{equation*}
    {B}_{2,\sigma,\mu}=\left\{\varphi_{i}, \psi_{i} \, :\, i=1,2,3\right\}
\end{equation*}
of the vector space $\mathbb{P}_{2}(T)$ satisfying the conditions
\begin{align}
{{\mathcal{I}}}_{j}^{\mathrm{enr}}\left(  \varphi_{i}\right) &=  \delta_{ij},\label{c1}
  \\ 
{{\mathcal{L}}}^{\mathrm{enr}}_{j}\left(\varphi_{i}\right) &=  0,  \label{c2} \\ 
{{\mathcal{I}}}_j^{\mathrm{enr}}\left(\psi_{i}\right) &= 0,  \label{c3}\\ {{\mathcal{L}}}^{\mathrm{enr}}_{j}\left(\psi_{i}\right) &=  \delta_{ij}, \label{c4}
\end{align}
for any $i,j=1,2,3$.

\begin{theorem}\label{th2allalf3r} The basis functions associated with the enriched triple $\mathcal{H}^{\mathrm{enr}}_{1,\sigma,\mu}$ are given by
\begin{align}
   \varphi_{i}&=  1-2\lambda_i,  \quad i=1,2,3, \label{varphi}\\
    \psi_{i}&=  - A_{\sigma,\mu}\varphi_{i}+\frac{2}{\left\|{p}_{\sigma,\mu}\right\|^2_{2,{{\sigma,\mu}} }}\left(-\lambda_i^2+\lambda^2_{i+1}+\lambda^2_{i+2}\right), \quad i=1,2,3, \label{phi}
\end{align}
where 
\begin{equation*}
    \left\|{p}_{\sigma,\mu}\right\|^2_{2,{{\sigma,\mu}} }=\int_{-1}^{1}{p}^2_{\sigma,\mu}(t)
\widetilde{k}_{\sigma,\mu}(t)dt,
\end{equation*}
\begin{equation}\label{Asigma}
A_{\sigma,\mu}=\frac{1+t_{2,\sigma,\mu}}{ \left\|{p}_{\sigma,\mu}\right\|^2_{2,{{\sigma,\mu}}}},
\end{equation}
and
\begin{equation*}
    t_{2,\sigma,\mu}=\int_{-1}^{1}t^2
\widetilde{k}_{\sigma,\mu}(t)dt,
\end{equation*}
denotes the second moment of $\ \widetilde{k}_{\sigma,\mu}.$
      \end{theorem}

We now have all the necessary ingredients to define the two-parameter local reconstruction operator associated with the pdfs  $\widetilde{k}_{\sigma,\mu}$.  
This operator is constructed using the basis functions of Theorem~\ref{th2allalf} and is defined by
\begin{equation}\label{pi1}
\begin{array}{rcl}
{\pi}_{1,\sigma,\mu}^{{\mathrm{enr}}}: C(T) &\rightarrow& {\mathbb{P}}_{2}(T)
\\
f &\mapsto& \displaystyle{\sum_{j=1}^{3}  {\mathcal{I}}^{\mathrm{enr}}_{j}\left(f\right){ \varphi}_{j}+ \sum_{j=1}^{3}{\mathcal{L}}^{\mathrm{enr}}_{j}\left(f\right)}{  \psi}_{j}.
\end{array}
\end{equation}

The proofs are based on certain properties of the pdfs $\widetilde{k}_{\sigma,\mu}$ and the polynomial $p_{\sigma,\mu}$.  
To establish Theorems~\ref{th1nnewfin11} and~\ref{th2allalf3r}, we first present several auxiliary lemmas and recall some concepts that will be useful in the sequel.  
In the discussion that follows, particular emphasis will be placed on the first three lemmas, which will be invoked repeatedly.\\

We begin with a lemma concerning the incomplete gamma function~\eqref{igamma}.
 
  \begin{lemma}\label{rho1} 
For any $\rho>0$ and $s>-1$, the following identity holds
 \begin{equation*}
       \int_{0}^{1} t^s e^{-(\rho  t)^2} dt =\frac{1}{2}  \gamma^{\mathrm{mod}}\left(\frac{s+1}{2},\rho^2\right),
     \end{equation*}
     where $\gamma^{\mathrm{mod}}$ denotes the modified incomplete gamma function defined in~\eqref{migamma}.
  \end{lemma}
  \begin{proof}
By using the change of variables 
$u = \rho^2 t^2$, we obtain
\begin{equation*}
   \int_{0}^{1} t^s e^{-(\rho t)^2} dt 
   = \frac{1}{2\rho^{s+1}} \int_{0}^{\rho^2} u^{\frac{s+1}{2}-1} e^{-u} du.
\end{equation*}
The result then follows directly from the definition~\eqref{migamma} with parameters $\frac{s+1}{2}$ and $\rho^2$.
  \end{proof}
  
 Exploiting the symmetry of the probability density $\widetilde{k}_{\sigma,\mu}$ defined in~\eqref{gwn=1},  
we first state a number of its fundamental properties.  
The next lemma shows that $\widetilde{k}_{\sigma,\mu}$ is indeed a probability density function,  
that is, it has unit zeroth moment.

\begin{lemma}\label{pd2}
For any $\mu \geq 1$ and $\sigma>0$, the function $\widetilde{k}_{\sigma,\mu}$ is a probability density function.
\end{lemma}

\begin{proof}
By symmetry and with the substitution $u = t^{\mu}$, we obtain
\begin{eqnarray}
  \int_{-1}^{1} \left(t^2\right)^{2\mu-2}  
   e^{-\frac{1}{2}\left(\frac{t^2}{\sigma^2}\right)^{\mu}} dt
   &=& 2 \int_{0}^{1} t^{4\mu-4}  
      e^{-\left(\frac{t^{\mu}}{\sqrt{2}\sigma^{\mu}}\right)^{2}} dt \notag \\ \label{iapaaa}
   &=& \frac{2}{\mu} \int_{0}^{1} 
      u^{\frac{3\mu-3}{\mu}} 
      e^{-\left(\frac{u}{\sqrt{2}\sigma^{\mu}}\right)^{2}} du.
\end{eqnarray}
Applying Lemma~\ref{rho1} with 
\begin{equation*}
s=\frac{3\mu-3}{\mu}, \quad     \rho=\frac{1}{\sqrt{2}\sigma^{\mu}},
\end{equation*}
we can write
\begin{equation}\label{sasasa1}
    \frac{2}{\mu} \int_{0}^{1} 
      u^{\frac{3\mu-3}{\mu}} 
      e^{-\left(\frac{u}{\sqrt{2}\sigma^{\mu}}\right)^{2}} du = \frac{1}{\mu}
   \gamma^{\mathrm{mod}}\left(\frac{4\mu-3}{2\mu},\frac{1}{2\sigma^{2\mu}}\right).
\end{equation}
Then, combining~\eqref{iapaaa} and~\eqref{sasasa1}, it follows that
\begin{equation}\label{n202522b}
     \int_{-1}^{1} \left(t^2\right)^{2\mu-2}  
   e^{-\frac{1}{2}\left(\frac{t^2}{\sigma^2}\right)^{\mu}} dt= \frac{1}{\mu} \gamma^{\mathrm{mod}}\left(\frac{4\mu-3}{2\mu},\frac{1}{2\sigma^{2\mu}}\right).
\end{equation}
Therefore, in order for $\widetilde{k}_{\sigma,\mu}$ to be a probability density function, the normalization constant $a_{\sigma,\mu}$ in~\eqref{moddens1} must be the reciprocal of the right-hand side of~\eqref{n202522b}, so that
\begin{equation*}
   \int_{-1}^{1} \widetilde{k}_{\sigma,\mu}(t) dt = 1. 
\end{equation*}
\end{proof}
  
The next lemma allows us to derive a compact formula for the moments 
of $\widetilde{k}_{\sigma,\mu}$ in terms of the modified incomplete gamma functions defined in~\eqref{migamma}.
  \begin{lemma}\label{mompd2}
  Let $\left\{t_{m,\sigma,\mu}\right\}_{m=0}^\infty$ be the sequence of moments of 
$\widetilde{k}_{\sigma,\mu}$. Then
     \begin{equation}\label{mom2}
     t_{m,\sigma,\mu}=  \int_{-1}^{1}t^m \widetilde{k}_{\sigma,\mu}(t) dt =\begin{cases}
			0, & \text{if $m=2k+1$ }\\
            \dfrac{   \gamma^{\mathrm{mod}}\left(\dfrac{2k+4\mu-3}{2\mu},\dfrac{1}{2\sigma^{2\mu}}\right)}{ \gamma^{\mathrm{mod}}\left(\dfrac{4\mu-3}{2\mu},\dfrac{1}{2\sigma^{2\mu}}\right)}, & \text{if $m=2k$}.
		 \end{cases}
     \end{equation}
  \end{lemma}
   \begin{proof}
 By symmetry, $\widetilde{k}_{\sigma,\mu}$ is even, hence for all $k\geq 0$, it follows that
\begin{equation*}
   t_{2k+1,\sigma,\mu} = 0.
\end{equation*}
For the even moments, performing the change of variables $u = t^{\mu}$ gives
 \begin{eqnarray}
 t_{2k,\sigma,\mu}&=&
\int_{-1}^{1} t^{2k}\left(t^2\right)^{2\mu-2} 
   e^{-\frac{1}{2} \left(\frac{t^2}{\sigma^2}\right)^{\mu}} dt\notag\\&=& 2 
    \int_{0}^{1} t^{2k +4\mu-4}   e^{-\left(\frac{t^{\mu}}{\sqrt{2}\sigma^{\mu}}\right)^2}  dt \notag\\ \label{n2025222} &=&
    \frac{2}{\mu}
    \int_{0}^{1} u^{\frac{2k+3\mu-3}{\mu}}   e^{-\left(\frac{u}{\sqrt{2}\sigma^{\mu}}\right)^2}  du. 
  \end{eqnarray}
Applying Lemma~\ref{rho1} with 
\begin{equation*}
s = \frac{2k+3\mu-3}{\mu}, \quad     \rho = \frac{1}{\sqrt{2}\sigma^{\mu}},
\end{equation*}
we can write
\begin{equation}\label{sasasa}
    \frac{2}{\mu}
    \int_{0}^{1} u^{\frac{2k+3\mu-3}{\mu}}   e^{-\left(\frac{u}{\sqrt{2}\sigma^{\mu}}\right)^2}  du=\frac{1}{\mu}
\gamma^{\mathrm{mod}}\left(\frac{2k+4\mu-3}{2\mu}, \frac{1}{2\sigma^{2\mu}}\right).
\end{equation}
Then, combining~\eqref{n2025222} and~\eqref{sasasa}, it follows that
  \begin{equation}
    \int_{-1}^{1} t^{2k}\left(t^2\right)^{2\mu-2} 
   e^{-\frac{1}{2} \left(\frac{t^2}{\sigma^2}\right)^{\mu}} dt=\frac{1}{\mu}
\gamma^{\mathrm{mod}}\left(\frac{2k+4\mu-3}{2\mu}, \frac{1}{2\sigma^{2\mu}}\right). \label{n20252221}
  \end{equation}
Then, by recalling the explicit expression of the pdfs $\widetilde{k}_{\sigma,\mu}$ in~\eqref{gwn=1} and using~\eqref{n20252221}, we obtain
 \begin{eqnarray*}
      t_{2k,\sigma,\mu}&=& \int_{-1}^1 t^{2k} \widetilde{k}_{\sigma,\mu}(t) dt\\&=&\frac{\mu }{    \gamma^{\mathrm{mod}}\left(\frac{4\mu-3}{2\mu},\frac{1}{2\sigma^{2\mu}}\right)}\int_{-1}^{1} t^{2k}\left(t^2\right)^{2\mu-2} 
   e^{-\frac{1}{2} \left(\frac{t^2}{\sigma^2}\right)^{\mu}} dt\\&=&  \frac{   \gamma^{\mathrm{mod}}\left(\frac{2k+4\mu-3}{2\mu},\frac{1}{2\sigma^{2\mu}}\right)}{ \gamma^{\mathrm{mod}}\left(\frac{4\mu-3}{2\mu},\frac{1}{2\sigma^{2\mu}}\right)}.
  \end{eqnarray*}
  \end{proof}
 \begin{remark} By setting $k=0$ in~\eqref{mom2}, one immediately recovers the normalization condition, 
thus confirming that $\widetilde{k}_{\sigma,\mu}(t)$ is indeed a probability density function.
  \end{remark}

\begin{remark}
The polynomial ${p}_{\sigma,\mu}$ defined in~\eqref{LegPol} 
can be expressed in terms of the second moment as
\begin{equation}\label{was}
   p_{\sigma,\mu}(t) = t^2 - t_{2,\sigma,\mu}.
\end{equation}
Indeed, by taking $m=2$ in identity~\eqref{mom2}, one has
\begin{equation*}
       t_{2,\sigma,\mu} 
   = \frac{\gamma^{\mathrm{mod}}\left(\frac{4\mu-1}{2\mu},\frac{1}{2\sigma^{2\mu}}\right)}
          {\gamma^{\mathrm{mod}}\left(\frac{4\mu-3}{2\mu},\frac{1}{2\sigma^{2\mu}}\right)},
\end{equation*}
which directly yields~\eqref{was}.
\end{remark}
The limiting moments of $\widetilde{k}_{\sigma,\mu}(t)$ as $\sigma \to \infty$ 
are easily obtained from the following remark.

\begin{remark}
We observe that
\begin{equation*}
   \lim_{\sigma\to \infty}\int_{-1}^1 t^m \widetilde{k}_{\sigma,\mu}(t) dt= \int_{-1}^1 t^m \widetilde{k}_{\infty,\mu}(t) dt.
\end{equation*}
Indeed, from the definition~\eqref{w1inf} we have
\begin{equation*}
 \int_{-1}^{1} t^m \widetilde{k}_{\infty,\mu}(t) dt= \frac{4\mu-3}{2}\int_{-1}^1 t^m\left(t^2\right)^{2\mu-2}dt. 
\end{equation*}
If $m$ is odd, this integral vanishes by symmetry. If instead $m=2k$, a straightforward calculation yields
\begin{equation*}
     \int_{-1}^{1} t^{2k} \widetilde{k}_{\infty,\mu}(t) dt= \frac{4\mu-3}{2}\int_{-1}^1 t^{2k}\left(t^2\right)^{2\mu-2}dt=\frac{4\mu-3}{2k+4\mu-3}. 
\end{equation*}
On the other hand, using~\eqref{mom2} together with~\eqref{modlimgamma}, we find
\begin{equation*}
    t_{2k,\infty,\mu}:=\lim_{\sigma\to\infty} t_{2k,\sigma,\mu} = \frac{4\mu-3}{2k+4\mu-3},
\end{equation*}
which are exactly the moments of order $2k$ associated with the pdfs  $\widetilde{k}_{\infty,\mu}(t)$.
\end{remark}

The following result shows that $p_{\sigma,\mu}$, defined in~\eqref{LegPol}, 
is indeed an orthogonal polynomial.

\begin{lemma}\label{poq} 
The polynomial $p_{\sigma,\mu}(t)$ is orthogonal to all linear polynomials on $[-1,1]$ 
with respect to the pdfs $\widetilde{k}_{\sigma,\mu}$.
\end{lemma}
\begin{proof}
Since $p_{\sigma,\mu}(t)\widetilde{k}_{\sigma,\mu}(t)$ is an even function, 
it suffices to show that
\begin{equation*}
\int_{-1}^1 p_{\sigma,\mu}(t)\widetilde{k}_{\sigma,\mu}(t)dt = 0.
\end{equation*}
Using the expression of $p_{\sigma,\mu}(t)$ from~\eqref{was}, we have
\begin{eqnarray*}
  \int_{-1}^1 p_{\sigma,\mu}(t) \widetilde{k}_{\sigma,\mu}(t) dt &=& \int_{-1}^1 \left( t^2 - t_{2,\sigma,\mu}\right) \widetilde{k}_{\sigma,\mu}(t)  dt \\
    &=&   \int_{-1}^1 t^2 \widetilde{k}_{\sigma,\mu}(t)dt - t_{2,\sigma,\mu} \int_{-1}^1 \widetilde{k}_{\sigma,\mu}(t)dt\\
    &=&0.  
\end{eqnarray*}
In the last equality we used the facts that
\begin{equation*}
    \int_{-1}^1 \widetilde{k}_{\sigma,\mu}(t) dt = 1,
\end{equation*}
since $\widetilde{k}_{\sigma,\mu}$ is a probability density function, and
\begin{equation*}
      \int_{-1}^1 t^2 \widetilde{k}_{\sigma,\mu}(t) dt = t_{2,\sigma,\mu},
\end{equation*}
which is precisely the second moment of $\widetilde{k}_{\sigma,\mu}$.
\end{proof}

We will need the average values of the barycentric coordinates on the edges of $T$, 
weighted by the probability density $\widetilde{k}_{\sigma,\mu}(t)$.

\begin{lemma}\label{I12}
For any $i,j=1,2,3$, it holds
\begin{equation}\label{ex32ssss}
{\mathcal{I}}_{j}^{\mathrm{enr}}\left(\lambda_i\right)
   = \frac{1}{2}\left(1-\delta_{ij}\right).
\end{equation}
\end{lemma}
\begin{proof}
From~\eqref{ex3qq}, we have
   \begin{equation*}
\mathcal{I}_{j}^{\mathrm{enr}}\left(\lambda_i\right)=\int_{-1}^1  \lambda_{i}\left(\frac{1+t}{2}\boldsymbol{v}_{j+1}+ \frac{1-t}{2} \boldsymbol{v}_{j+2}\right) \widetilde{k}_{\sigma,\mu}(t) dt.
   \end{equation*}
If $i=j$, then by~\eqref{propstar} it follows immediately that
\begin{equation*}
{\mathcal{I}}_{j}^{\mathrm{enr}}\left(\lambda_{j}\right)=0, \quad j=1,2,3.
\end{equation*}
If instead $i \neq j$, then one of the values 
$\lambda_i(\boldsymbol{v}_{j+1})$ or $\lambda_i(\boldsymbol{v}_{j+2})$ equals $1$, 
while the other equals $0$. Since $\widetilde{k}_{\sigma,\mu}$ is a probability density function  
and barycentric coordinates are affine functions, we get
    \begin{align*}
{\mathcal{I}}_{j}^{\mathrm{enr}}\left(\lambda_i\right)&=\int_{-1}^1  \left(\frac{1+t}{2}\lambda_i\left(\boldsymbol{v}_{j+1}\right)+ \frac{1-t}{2} \lambda_{i}\left(\boldsymbol{v}_{j+2}\right)\right) \widetilde{k}_{\sigma,\mu}(t) dt \notag\\
&= \frac{1}{2}\int_{-1}^1 \left((1+t)\lambda_i\left(\boldsymbol{v}_{j+1}\right)+(1-t) \lambda_{i}\left(\boldsymbol{v}_{j+2}\right)\right) \widetilde{k}_{\sigma,\mu}(t) dt\\
&=  \frac{1}{2}\int_{-1}^1 \widetilde{k}_{\sigma,\mu}(t) dt=\frac{1}{2},
    \end{align*}
where in the last equality we have used the fact that  $\widetilde{k}_{\sigma,\mu}$ is a probability density.
 Then the lemma is proved.
\end{proof}

We also need explicit expressions for the average values of the square of the barycentric coordinates
on each edge of $T$, weighted by $\widetilde{k}_{\sigma,\mu}(t)$.
\begin{lemma}\label{I22}
  For any $i,j=1,2,3$, it holds 
\begin{equation}\label{ex32as2}
{\mathcal{I}}_{j}^{\mathrm{enr}}\left(\lambda_i^2\right)=\frac{1+t_{2,\sigma,\mu} }{4}
 \left(1-\delta_{ij}\right),
\end{equation} 
where $t_{2,\sigma,\mu}$ is the second moment  of $\widetilde{k}_{\sigma,\mu}$.
\end{lemma} 
\begin{proof}
From~\eqref{ex3qq}, we have
\begin{equation*}
{\mathcal{I}}_{j}^{\mathrm{enr}}\left(\lambda_i^2\right)=\int_{-1}^1  \lambda_{i}^2\left(\frac{1+t}{2}\boldsymbol{v}_{j+1}+ \frac{1-t}{2} \boldsymbol{v}_{j+2}\right) \widetilde{k}_{\sigma,\mu}(t) dt.
    \end{equation*} 
If $i=j$, then by~\eqref{propstar}, it follows immediately that 
\begin{equation*}
{\mathcal{I}}_{j}^{\mathrm{enr}}\left(\lambda_{j}^2\right)=0, \quad j=1,2,3.
\end{equation*}
If instead $i \neq j$, we obtain
    \begin{align}
{\mathcal{I}}_{j}^{\mathrm{enr}}\left(\lambda_i^2\right)&=\int_{-1}^1 \left(\frac{1+t}{2}\lambda_i\left(\boldsymbol{v}_{j+1}\right)+ \frac{1-t}{2} \lambda_{i}\left(\boldsymbol{v}_{j+2}\right)\right)^2 \widetilde{k}_{\sigma,\mu}(t) dt \notag\\
&= \frac{1}{4}\int_{-1}^1 \left((1+t)\lambda_i\left(\boldsymbol{v}_{j+1}\right)+(1-t) \lambda_{i}\left(\boldsymbol{v}_{j+2}\right)\right)^2 \widetilde{k}_{\sigma,\mu}(t) dt. \label{ulteqs2}
    \end{align}
Proceeding as in the proof of Lemma~\ref{I12}, the integral in~\eqref{ulteqs2} decomposes as
    \begin{equation*}
{\mathcal{I}}_{j}^{\mathrm{enr}}\left(\lambda_i^2\right)=\frac{1}{4}\int_{-1}^1 \widetilde{k}_{\sigma,\mu}(t) dt+\frac{1}{4}\int_{-1}^1 t^2 \widetilde{k}_{\sigma,\mu}(t)dt.
    \end{equation*}
    Since $\widetilde{k}_{\sigma,\mu}$ is a probability density, the first integral equals $1$,
while the second equals $t_{2,\sigma,\mu}$. This proves~\eqref{ex32as2}.
\end{proof}

The following result is an immediate consequence of Lemma~\ref{poq}.

\begin{lemma}\label{lem22}
For any $i,j=1,2,3$, the following identity holds
\begin{equation}\label{F12}
{\mathcal{L}}_{j}^{\mathrm{enr}}\left(\lambda_i\right)=0.
\end{equation}
\end{lemma}
\begin{proof}
The result follows directly from the definition of 
$\mathcal{L}_{j}^{\mathrm{enr}}$ given in~\eqref{ex3qqq}, together with the 
orthogonality property of the polynomial $p_{\sigma,\mu}(t)$ established 
in Lemma~\ref{poq}.
\end{proof}

\begin{remark}
From the definition of $\mathcal{L}_{j}^{\mathrm{enr}}$ and the 
orthogonality of the polynomial $p_{\sigma,\mu}$ established in 
Lemma~\ref{poq}, it follows that
\begin{equation*}
\mathcal{L}_{j}^{\mathrm{enr}}\left(p_1\right)=0, 
\quad j=1,2,3,
\end{equation*}
for every linear polynomial $p_1 \in \mathbb{P}_{1}(T)$.
\end{remark}

\begin{lemma}\label{F22}
For any $i,j=1,2,3$, the following identity holds
\begin{equation}\label{ex3212}
{\mathcal{L}}_{j}^{\mathrm{enr}}\left(\lambda_i^2\right)= 
\frac{\left\|p_{\sigma,\mu}\right\|^2_{2,{{\sigma,\mu}}}}{4} \left(1-\delta_{ij}\right).
\end{equation}  
\end{lemma} 
\begin{proof}
From equation~\eqref{ex3qqq} and the fact that barycentric coordinates are affine functions, we obtain
\begin{align*}
{\mathcal{L}}_{j}^{\mathrm{enr}}\left(\lambda_{i}^2\right)&= \int_{-1}^{1} p_{\sigma,\mu}(t)
  \lambda_i^2\left(\frac{1+t}{2}\B{v}_{j+1} + \frac{1-t}{2}\B{v}_{j+2}\right)\widetilde{k}_{\sigma,\mu}(t)dt \notag\\ &= \int_{-1}^{1} p_{\sigma,\mu}(t)
  \left(\frac{1+t}{2}\lambda_i\left(\B{v}_{j+1}\right) + \frac{1-t}{2} \lambda_i\left(\B{v}_{j+2}\right)\right)^2 \widetilde{k}_{\sigma,\mu}(t)dt.
\end{align*}
If $i=j$, then by~\eqref{propstar} it follows that  
\begin{equation*}
{\mathcal{L}}_{j}^{\mathrm{enr}}\left(\lambda_{j}^2\right)=0, \quad j=1,2,3.
\end{equation*}
If instead $i\neq j$, then $\lambda_i\left(\B{v}_{j+1}\right)=1$ or $\lambda_i\left(\B{v}_{j+2}\right)=1$, and one of them must be zero. 
Using this fact, and exploiting the orthogonality 
property of $p_{\sigma,\mu}$ established in Lemma~\ref{poq}, together with the identity~\eqref{was} expressing $t^2$ in terms of $p_{\sigma,\mu}$, we obtain
\begin{align*} {\mathcal{L}}_{j}^{\mathrm{enr}}\left(\lambda_{i}^2\right)&=  \frac{1}{4}\int_{-1}^{1} p_{\sigma,\mu}(t) t^2\widetilde{k}_{\sigma,\mu}(t)dt\\ 
&=\frac{1}{4}\int_{-1}^{1} p_{\sigma,\mu}(t) \left(p_{\sigma,\mu}(t)+t_{2,\sigma,\mu} \right) 
\widetilde{k}_{\sigma,\mu}(t)dt.
\end{align*}
Using again the orthogonality property of $p_{\sigma,\mu}$
the desired result follows.
\end{proof}

We are now ready for the proof that the triple  $\mathcal{H}^{\mathrm{enr}}_{1,\sigma,\mu}$ is unisolvent. To this aim, we first introduce some notations that will be used 
throughout the remainder of the paper. In particular, we consider the matrix
\begin{equation}\label{MA}
A:=  \begin{bmatrix}
0 & 1 & 1\\
1 & 0& 1\\
1 & 1 & 0
\end{bmatrix}.
\end{equation}
As $\det(A)=2\neq 0$, the matrix $A$ is invertible, with inverse given by
\begin{equation}\label{IMA}
A^{-1}= \frac{1}{2} \begin{bmatrix}
-1 & 1 & 1\\
1 & -1 & 1\\
1 & 1 & -1
\end{bmatrix}.
\end{equation}

\noindent
\textbf{Proof of Theorem~\ref{th1nnewfin11}.} 
It is sufficient to prove that the only polynomial 
$p_2\in \mathbb{P}_{2}(T)$ satisfying
\begin{align}
\mathcal{I}_{j}^{\mathrm{enr}}\!\left(p_2\right)&=0, \quad j=1,2,3, \label{conds2}\\  
\mathcal{L}_{j}^{\mathrm{enr}}\!\left(p_2\right)&=0, \quad j=1,2,3, \label{conds12}
\end{align} 
is the zero polynomial. To this aim, let $p_2\in\mathbb{P}_2(T)$, which can be expressed in barycentric form as
\begin{equation}\label{polptild}
p_2=\sum_{i=1}^3a_i \lambda_i+\sum_{i=1}^3 b_i \lambda_i^2,   
 \end{equation}
with coefficients $a_i,b_i\in\mathbb{R}$, $i=1,2,3$.
By Lemmas~\ref{lem22} and~\ref{F22}, condition~\eqref{conds12} reduces to
\begin{align*}
0 = {\mathcal{L}}^{\mathrm{enr}}_{j}\left(p_2\right) &=\sum_{\substack{i=1}}^{3} b_{i} \mathcal{L}^{\mathrm{enr}}_{j}\left(\lambda^2_i\right)=\frac{ \left\|p_{\sigma,\mu}\right\|^2_{{{2,\sigma,\mu}}} }{4}\sum_{\substack{i=1}}^{3} b_{i}\left(1-\delta_{ij}\right).
\end{align*}
This can equivalently be written in matrix form as
\begin{equation*}
 \dfrac{ \left\|p_{\sigma,\mu}\right\|^2_{2,{{\sigma,\mu}}} }{4} A{\B b}={\B 0},
\end{equation*}
where $A$ is the matrix defined in~\eqref{MA}, and
 \begin{equation*}
     {\B b}^T=\left(b_1,b_2,b_3\right), \quad {\B 0}^T=(0,0,0).
 \end{equation*}
Since $A$ is non-singular, we conclude that
\begin{equation*}
    b_1=b_2=b_3=0.
\end{equation*}
Consequently, by~\eqref{polptild}, the polynomial $p_2$ reduces to
\begin{equation*}
p_2=p_1= a_1\lambda_1+a_2\lambda_2+a_3\lambda_3, \quad a_1,a_2,a_3\in\mathbb{R}.
\end{equation*}
Applying Lemma~\ref{I12}, condition~\eqref{conds2} becomes
\begin{align*}
0 &={\mathcal{I}}_{j}^{\mathrm{enr}}\left(p_2\right) = \frac{1}{2}\sum_{\substack{i=1}}^{3} a_{i}\left(1-\delta_{ij}\right).
\end{align*}
This can equivalently be written in matrix form as
\begin{equation}\label{mf2}
\frac{1}{2} A{\B a}={\B 0},
\end{equation}
 where 
 \begin{equation*}
 {\B a}^T=\left(a_1,a_2,a_3\right). 
 \end{equation*}
Since $A$ is non-singular, system~\eqref{mf2} admits only the trivial solution
\begin{equation*}
    a_1=a_2=a_3=0,
\end{equation*}
which implies $p_2=0$ and thus completes the proof. \qed

The proof of Theorem~\ref{th2allalf} relies on the derivation of the explicit
analytic expressions of the basis functions $\varphi_{i}$ and $\psi_{i}$.
We illustrate the procedure by computing $\varphi_{1}$ and $\psi_{1}$,
as the remaining cases follow in the same way.\\

\noindent
\textbf{Proof of Theorem~\ref{th2allalf3r}. }\\  
We begin by deriving the analytic expression of the basis function $\varphi_{1}$. 
According to its definition, $\varphi_{1}$ can be expressed as 
\begin{equation}\label{phi1old2}
\varphi_{1} = \sum_{i=1}^3 a_{i}\lambda_i + \sum_{i=1}^3 b_{i}\lambda_i^2,
\end{equation}
for some coefficients $a_i,b_i\in\mathbb{R}$, $i=1,2,3$. 
To determine these coefficients, we use condition~\eqref{c2} together with 
Lemmas~\ref{lem22} and~\ref{F22}, which yield
\begin{equation*}
    0 = \mathcal{L}^{\mathrm{enr}}_{j}\left(\varphi_{1}\right)
  = \sum_{i=1}^{3} b_{i}\mathcal{L}^{\mathrm{enr}}_{j}\left(\lambda^2_i\right) = \frac{\|p_{\sigma,\mu}\|^2_{2,{\sigma,\mu}}}{4}\sum_{i=1}^{3} b_{i}\left(1-\delta_{ij}\right).
\end{equation*}
As in the previous proof, this system admits only the trivial solution
\begin{equation*}
b_1=b_2=b_3=0.
\end{equation*}
Substituting these values into~\eqref{phi1old2}, $\varphi_{1}$ reduces to
\begin{equation*}
\varphi_{1}=\sum_{i=1}^3 a_{i}\lambda_i.
\end{equation*}
Applying Lemma~\ref{I12}, condition~\eqref{c1} becomes
\begin{align*}
\delta_{1j} &={\mathcal{I}}^{\mathrm{enr}}_{j}\left(\varphi_{1}\right) =\sum_{\substack{i=1}}^{3} a_{i} {\mathcal{I}}^{\mathrm{enr}}_{j}\left(\lambda_i\right)=\frac{1}{2} \sum_{\substack{i=1}}^{3} a_{i}\left(1-\delta_{ij}\right).
\end{align*}
This is equivalently written in matrix form as
\begin{equation*}
    \left(1,0,0\right)^T=\frac{1}{2}A  {\B a}
\end{equation*}
where 
 \begin{equation*}
 {\B a}^T=\left(a_1,a_2,a_3\right). 
 \end{equation*}
A straightforward computation yields
\begin{equation*}
     a_1=-1, \quad a_2=a_3=1,
\end{equation*}
and hence
\begin{equation*}
\varphi_{1} = 1-2\lambda_1.
\end{equation*}
It remains to compute the expression of the basis function $\psi_{1}$. 
By definition, it can be expressed as
\begin{equation}\label{phi1}
\psi_{1} = \sum_{i=1}^3 c_{i}\lambda_i + \sum_{i=1}^3 d_{i} \lambda_i^2, 
\end{equation}
for some coefficients $c_i,d_i\in\mathbb{R}$, $i=1,2,3$. 
To determine these coefficients, we use condition~\eqref{c4} together with 
Lemmas~\ref{lem22} and~\ref{F22}, which yield
\begin{align*}
\delta_{1j} &= {\mathcal{L}}^{\mathrm{enr}}_{j}\left(\psi_{1}\right)= \sum_{\substack{i=1}}^{3} d_{i} {\mathcal{L}}^{\mathrm{enr}}_{j}\left(\lambda_i^2\right)=\frac{ \left\|p_{\sigma,\mu}\right\|^2_{2,{{\sigma,\mu}}} }{4}
\sum_{\substack{i=1}}^{3} d_{i}\left(1-\delta_{ij}\right).
\end{align*}
This is equivalently written in matrix form as
\begin{equation*}
    \left(1,0,0\right)^T=\frac{ \left\|p_{\sigma,\mu}\right\|^2_{2,{{\sigma,\mu}}} }{4} A  {\B d}
\end{equation*}
where 
 \begin{equation*}
 {\B d}^T=\left(d_1,d_2,d_3\right). 
 \end{equation*}
Solving this system, we have
\begin{equation*}
    d_1=-\frac{2}{\left\|p_{\sigma,\mu}\right\|^2_{2,{{\sigma,\mu}} }}, \quad d_2=d_3=\frac{2}{\left\|p_{\sigma,\mu}\right\|^2_{2,{{\sigma,\mu}} }}.
\end{equation*}
Substituting these values into~\eqref{phi1}, $\psi_{1}$ reduces to
\begin{equation*} 
\psi_{1}=\sum_{i=1}^3 c_{i} \lambda_i +\frac{2}{\left\|p_{\sigma,\mu}\right\|^2_{2,{{\sigma,\mu}} }}\left(-\lambda_1^2+\lambda^2_2+\lambda^2_3\right). 
\end{equation*}
Applying Lemmas~\ref{I12} and~\ref{I22}, condition~\eqref{c3} yields
\begin{align*}
0 &=  {\mathcal{I}}^{\mathrm{enr}}_{j}\left(\psi_{1}\right) =\frac{1}{2} \sum_{\substack{i=1}}^{3} c_{i}\left(1-\delta_{ij}\right)+A_{\sigma,\mu}\left(1-\delta_{ij}\right),
\end{align*}
where $A_{\sigma,\mu}$ is defined in~\eqref{Asigma}.  
This system admits the unique solution
\begin{equation*}
     c_{1}=A_{\sigma,\mu}, \quad c_2=c_3=-A_{\sigma,\mu}. 
\end{equation*}
Hence $\psi_{1}$ takes the form
\begin{equation*}
        \psi_{1}= - A_{\sigma,\mu}\varphi_{1}+\frac{2}{\left\|p_{\sigma,\mu}\right\|^2_{2,{{\sigma,\mu}} }}\left(-\lambda_1^2+\lambda^2_2+\lambda^2_3\right),
\end{equation*}
which concludes the proof.
\qed

\subsection{Second family of distributions}\label{ss2}
We now present a second two-parameter family of exponential-type distributions, which generalizes the truncated normal distribution, and is given by the following bivariate function with parameters $\mu \geq 1$ and $\sigma > 0$
 \begin{equation}\label{êxp3}
  G_{\sigma,\mu}(\B x) =  \frac{(2\mu-1)}{ \gamma^{\mathrm{mod}}\left(\frac{1}{2}, \frac{1}{2\sigma^{4\mu-2}} \right)}
 H^{\mu -1}(\B x)
 e^{ -\frac{1}{2}   \left(\frac{H(\B x)}{\sigma^2} \right)^{2\mu-1} },
\end{equation}
where $H$ is given by~\eqref{gw}. Unlike the function of the first family,
$K_{\sigma,\mu}$ in~\eqref{exp2}, the function $G_{\sigma,\mu}$ does not, in general,
define a weight function for all values of $\mu$, because $H$ changes sign on $T$. In fact, at the barycenter $$\displaystyle \bar{\B v}:=\frac{1}{3}\left(\B v_1+\B v_2+\B v_3\right)$$
one has $H(\bar{\B v})=-\frac{1}{3}$, hence
$$H^{\mu-1}(\bar{\B v})=\left(-\frac{1}{3}\right)^{\mu-1},$$ which is negative whenever $\mu$ is an integer and
$\mu-1$ is odd. 

With the normalization in~\eqref{êxp3}, the restriction of $G_{\sigma,\mu}$ to any edge
of $T$ yields a \emph{probability density function} on $[-1,1]$ (see Lemma~\ref{d1}).
Indeed, using~\eqref{prop2} and~\eqref{propstar} we obtain
\begin{align}
 \widetilde{g}_{\sigma,\mu}(t)
 &= G_{\sigma,\mu}\left(\frac{1+t}{2}\B v_{j+1} + \frac{1-t}{2}\B v_{j+2}\right) \nonumber\\
 &= \frac{2\mu-1}{\gamma^{\mathrm{mod}}\left(\frac{1}{2},\frac{1}{2\sigma^{4\mu-2}}\right)}
    \left(t^{2}\right)^{\mu-1}
    e^{-\frac{1}{2}\left(\frac{t^{2}}{\sigma^{2}}\right)^{2\mu-1}},
 \quad t\in[-1,1]. \label{gwn=2}
\end{align}
\begin{remark}
When $\mu=1$, the edge-restricted densities $\widetilde{g}_{\sigma,1}(t)$ and $\widetilde{k}_{\sigma,1}(t)$
coincide for $t\in[-1,1]$. In particular, both reduce to the standard truncated normal distribution on $[-1,1]$.
\end{remark}
  \begin{figure}
      \centering
\includegraphics[width=0.49\linewidth]{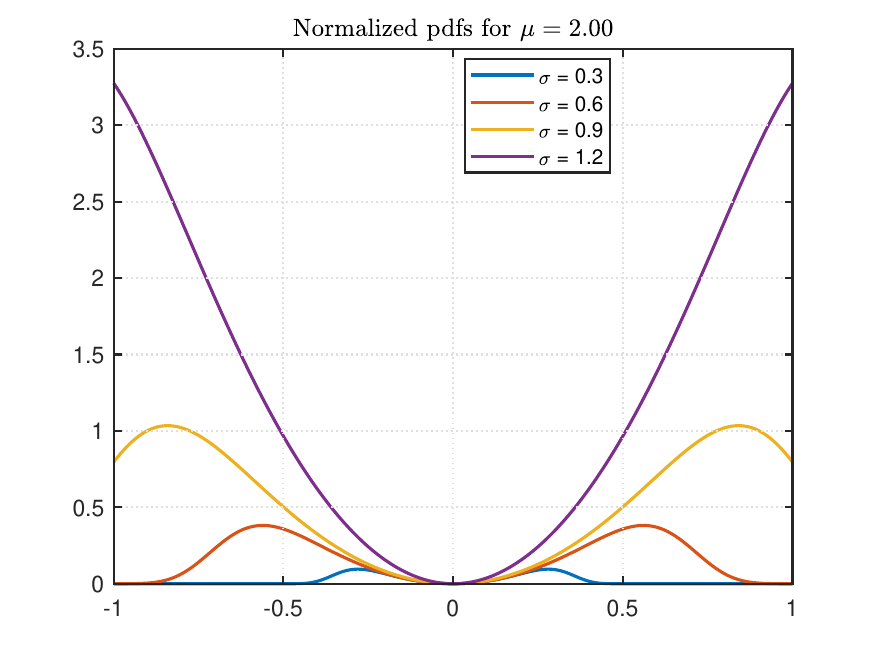}
\includegraphics[width=0.49\linewidth]{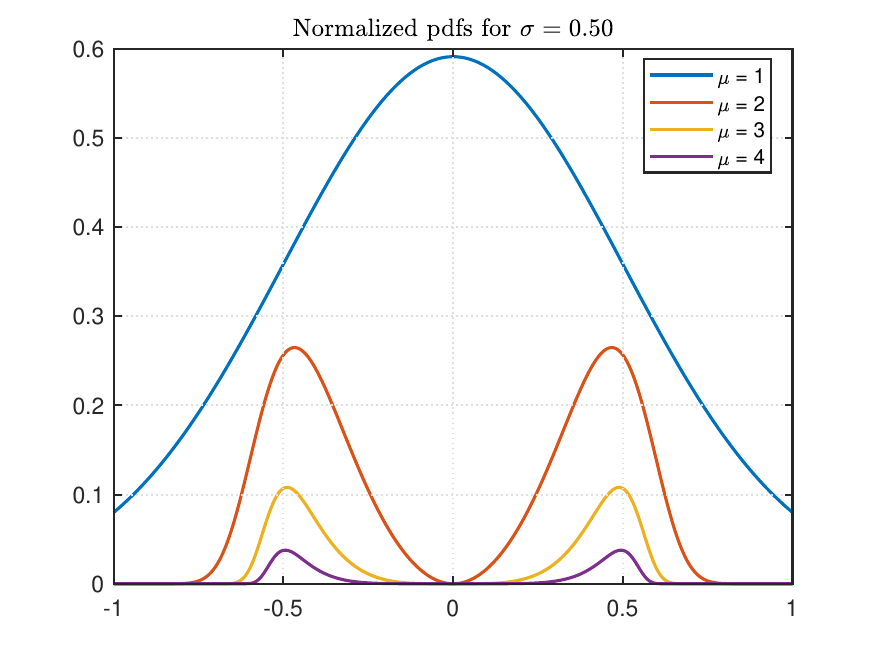}
      \caption{Normalized probability density functions $\widetilde{g}_{\sigma,\mu}$ on $[-1,1]$. 
Left: effect of the scale parameter, with fixed $\mu=2$ and different values of $\sigma$. 
Right: effect of the shape parameter, with fixed $\sigma=0.50$ and varying values of $\mu$, illustrating the increased sharpness and concentration around $t=0$. 
These plots emphasize the flexibility provided by the two parameters.}
      \label{fig:densities2}
  \end{figure}
Figure~\ref{fig:densities2} shows that $\widetilde{g}_{\sigma,\mu}$ depends strongly on both the shape parameter $\mu$ and the scale parameter $\sigma$, thereby motivating the two-parameter generalization.

To define our second two-parameter family of nonconforming histopolation operators,
we introduce the following enriched linear functionals
\begin{align}\label{ex31}
{\mathcal{D}}_{j}^{\mathrm{enr}}(f) &= \int_{-1}^{1} f\left(\frac{1+t}{2}\B{v}_{j+1} + \frac{1-t}{2}\B{v}_{j+2}\right) \widetilde{g}_{\sigma,\mu} (t) dt, \quad j=1,2,3,\\[2mm]
{\mathcal{S}}_{j}^{\mathrm{enr}}\left(f\right)&=
 \int_{-1}^{1} q_{\sigma,\mu}(t)
  f\left(\frac{1+t}{2}\B{v}_{j+1} + \frac{1-t}{2}\B{v}_{j+2}\right) \widetilde{g}_{\sigma,\mu} (t) dt, \quad j=1,2,3. \label{ex3}
\end{align}
 where 
\begin{equation}\label{LegPol2}
{q}_{\sigma,\mu}(t)=     t^2- \frac{\gamma^{\mathrm{mod}}\left(\frac{2\mu+1}{2(2\mu-1)},\frac{1}{2\sigma^{4\mu-2}}\right)}
          {\gamma^{\mathrm{mod}}\left(\frac{1}{2},\frac{1}{2\sigma^{4\mu-2}}\right)}.
\end{equation}
We then define the following triple
\begin{equation*}
\mathcal{H}^{\mathrm{enr}}_{2,\sigma,\mu}= \left(T, \mathbb{P}_{2}(T),\Sigma_{2,\sigma,\mu}^{\mathrm{enr}}(T)\right),
\end{equation*}
where  
 \begin{equation*}
\Sigma^{\mathrm{enr}}_{2,\sigma,\mu}(T)=  \left\{{\mathcal{D}}_{j}^{\mathrm{enr}},{{\mathcal{S}}}^{\mathrm{enr}}_{j}\, : \, j=1,2,3\right\}.
\end{equation*}

 \begin{remark}
The distribution $\widetilde{g}_{\sigma,\mu}$ admits a beta-type (power) density as a limiting case.
Indeed, using~\eqref{limgamma} with
 \begin{equation*}
s =\frac{1}{2}, \quad  z= \frac{1}{2\sigma^{4\mu-2}}    
 \end{equation*}
it follows that
\begin{equation*}
  \lim_{z \rightarrow 0^+ } \gamma^{\mathrm{mod}}\left(\frac{1}{2},z\right)= \lim_{z \rightarrow 0^+ }\frac{\gamma\left(\frac{1}{2},z\right) }{\sqrt{z}} = 2.
\end{equation*}
Therefore, by~\eqref{gwn=2}, for all $t\in[-1,1]$ we obtain the pointwise limit
\begin{equation}\label{w1inf2}
\widetilde{g}_{\infty,\mu}(t):=  \lim_{\sigma \rightarrow \infty } \widetilde{g}_{\sigma,\mu}(t)= \frac{2\mu-1}{2} \left(t^2\right)^{\mu -1}.
\end{equation}
Hence, $\widetilde{g}_{\infty,\mu}$ is the symmetric beta-type (power) density on $[-1,1]$,
highlighting the limiting behaviour of the family as $\sigma\to\infty$.
 \end{remark}

Our first goal is to prove that $\mathcal{H}_{2,\sigma,\mu}^{\mathrm{enr}}$ is unisolvent.
In a second step, we characterize the basis functions associated with this triple.
Proofs are deferred to later subsections. The main results and some immediate
consequences are summarized below.

\begin{theorem}\label{th1nnewfin112} The triple $\mathcal{H}^{\mathrm{enr}}_{2,\sigma,\mu}$ is unisolvent for any $\sigma>0$ and $\mu\geq 1$.
\end{theorem}
We next derive closed-form expressions for the basis functions associated with the triple $\mathcal{H}^{\mathrm{enr}}_{2,\sigma,\mu}$.  
Specifically, we determine a basis
\begin{equation*}
{B}^{\prime}_{2,\sigma,\mu}=\left\{\eta_{i}, \rho_{i} \, :\, i=1,2,3\right\}
\end{equation*}
of the vector space $\mathbb{P}_{2}(T)$ satisfying the conditions
\begin{align*}
{{\mathcal{D}}}_{j}^{\mathrm{enr}}\left(  \eta_{i}\right) &=  \delta_{ij},
  \\ 
{{\mathcal{S}}}^{\mathrm{enr}}_{j}\left(\eta_{i}\right) &=  0,  \\ 
{\mathcal{D}}_j^{\mathrm{enr}}\left(\rho_i\right) &= 0,  \\ {{\mathcal{S}}}^{\mathrm{enr}}_{j}\left(\rho_{i}\right) &=  \delta_{ij}, 
\end{align*}
for any $i,j=1,2,3$. 

\begin{theorem}\label{th2allalf2} The basis functions associated with the enriched triple $\mathcal{H}^{\mathrm{enr}}_{2,\sigma,\mu}$ are given by
\begin{align*}
    \eta_{i}&=  1-2\lambda_i,  \quad i=1,2,3,\\
    \rho_{i}&=  - C_{\sigma,\mu}\eta_{i}+\frac{2}{\left\|{q}_{\sigma,\mu}\right\|^2_{2,{{\sigma,\mu}} }}\left(-\lambda_i^2+\lambda^2_{i+1}+\lambda^2_{i+2}\right), \quad i=1,2,3, 
\end{align*} 
where 
\begin{equation*}
    \left\|{q}_{\sigma,\mu}\right\|^2_{2,{{\sigma,\mu}} }=\int_{-1}^{1}{q}^2_{\sigma,\mu}(t)
\widetilde{g}_{\sigma,\mu}(t)dt,
\end{equation*}
\begin{equation*}
C_{\sigma,\mu}=\frac{1+s_{2,\sigma,\mu}}{ \left\|q_{\sigma,\mu}\right\|^2_{2,{{\sigma,\mu}}}},
\end{equation*}
and
\begin{equation*}
    s_{2,\sigma,\mu}=\int_{-1}^{1}t^2
\widetilde{g}_{\sigma,\mu}(t)dt,
\end{equation*}
denotes the second moment of $\ \widetilde{g}_{\sigma,\mu}.$
      \end{theorem}

We now have all the necessary ingredients to define the two-parameter local reconstruction operator associated with the pdfs  $\widetilde{g}_{\sigma,\mu}$.  
This operator is constructed using the basis functions of Theorem~\ref{th2allalf2} and is defined by
\begin{equation*}
\begin{array}{rcl}
{\pi}_{2,\sigma,\mu}^{{\mathrm{enr}}}: C(T) &\rightarrow& {\mathbb{P}}_{2}(T)
\\
f &\mapsto& \displaystyle{\sum_{j=1}^{3}  {\mathcal{D}}^{\mathrm{enr}}_{j}\left(f\right){ \eta}_{j}+ \sum_{j=1}^{3}{\mathcal{S}}^{\mathrm{enr}}_{j}\left(f\right)}{  \rho}_{j}.
\end{array}
\end{equation*}

As in the previous case, the proofs rely on specific properties of the density
$\widetilde{g}_{\sigma,\mu}$ and of the polynomial $q_{\sigma,\mu}$. To prove
Theorems~\ref{th1nnewfin112} and~\ref{th2allalf2}, we first collect a number of auxiliary lemmas
and recall some standard notions needed in the sequel. Exploiting the symmetry of the edge density
$\widetilde{g}_{\sigma,\mu}$ defined in~\eqref{gwn=2}, we begin by recording some of its basic
properties. The next lemma shows that $\widetilde{g}_{\sigma,\mu}$ is a probability density on
$[-1,1]$.

\begin{lemma}\label{d1}
For any $\mu \geq 1$ and $\sigma>0$, the function $\widetilde{g}_{\sigma,\mu}$ is a probability density function.
\end{lemma}

\begin{proof}
By symmetry and with the substitution $u = t^{2\mu-1}$, we obtain
\begin{eqnarray}
  \int_{-1}^{1} \left(t^2\right)^{\mu-1}
   e^{-\frac{1}{2}\left(\frac{t^2}{\sigma^2}\right)^{2\mu-1}}dt
   &=& 2 \int_{0}^{1} t^{2\mu-2}
      e^{-\left(\frac{t^{2\mu-1}}{\sqrt{2}\sigma^{2\mu-1}}\right)^{2}} dt \notag \\ \label{gs1}
   &=& \frac{2}{2\mu-1} \int_{0}^{1}
      e^{-\left(\frac{u}{\sqrt{2}\sigma^{2\mu-1}}\right)^{2}} du.
\end{eqnarray}
Applying Lemma~\ref{rho1} with
\begin{equation*}
    s=0, \quad \rho=\frac{1}{\sqrt{2}\sigma^{2\mu-1}},
\end{equation*}
we can write
\begin{equation}\label{gs2}
    \frac{2}{2\mu-1} \int_{0}^{1}
      e^{-\left(\frac{u}{\rho}\right)^{2}} du
    = \frac{1}{2\mu-1}
      \gamma^{\mathrm{mod}}\left(\frac{1}{2},\frac{1}{2\sigma^{4\mu-2}}\right).
\end{equation}
Then, combining~\eqref{gs1} and~\eqref{gs2}, it follows that
\begin{equation}\label{gs3}
     \int_{-1}^{1} \left(t^2\right)^{\mu-1}
     e^{-\frac{1}{2}\left(\frac{t^2}{\sigma^2}\right)^{2\mu-1}} dt
     = \frac{1}{2\mu-1}
       \gamma^{\mathrm{mod}}\left(\frac{1}{2},\frac{1}{2\sigma^{4\mu-2}}\right).
\end{equation}
Therefore, by the definition of $\widetilde{g}_{\sigma,\mu}$ in~\eqref{gwn=2}, whose normalization
constant equals 
\begin{equation*}
 \frac{2\mu-1}{\gamma^{\mathrm{mod}}\left(\frac{1}{2},\frac{1}{2\sigma^{4\mu-2}}\right)},   
\end{equation*}
 we get
\begin{equation*}
       \int_{-1}^{1} \widetilde{g}_{\sigma,\mu}(t)dt = 1.
\end{equation*}
This proves the claim.
\end{proof}

The next lemma allows us to derive a compact formula for the moments 
of $\widetilde{g}_{\sigma,\mu}$ in terms of the modified incomplete gamma functions defined in~\eqref{migamma}.

\begin{lemma}\label{momw1}
Let $\{s_{m,\sigma,\mu}\}_{m=0}^\infty$ be the sequence of normalized moments of $\widetilde{g}_{\sigma,\mu}$. Then
\begin{equation}\label{dsssaasasa}
s_{m,\sigma,\mu}=\int_{-1}^{1} t^{m}\widetilde{g}_{\sigma,\mu}(t)dt
=\begin{cases}
0, & \text{if } m=2k+1,\\[1mm]
\dfrac{\gamma^{\mathrm{mod}}\left(\dfrac{2k+2\mu-1}{2(2\mu-1)},\dfrac{1}{2\sigma^{4\mu-2}}\right)}
{\gamma^{\mathrm{mod}}\left(\dfrac{1}{2},\dfrac{1}{2\sigma^{4\mu-2}}\right)}, & \text{if } m=2k.
\end{cases}
\end{equation}
\end{lemma}

\begin{proof}
 By symmetry, $\widetilde{g}_{\sigma,\mu}$ is even, hence for all $k\geq 0$, it follows that
\begin{equation*}
   s_{2k+1,\sigma,\mu} = 0.
\end{equation*}
For the even moments, performing the change of variables  $u=t^{2\mu-1}$ gives
 \begin{eqnarray}
s_{2k,\sigma,\mu}&=& \frac{2\mu-1}{\gamma^{\mathrm{mod}}\left(\frac{1}{2},\frac{1}{2\sigma^{4\mu-2}}\right)}\int_{-1}^{1} t^{2k}\left(t^2\right)^{\mu-1} 
   e^{-\frac{1}{2} \left(\frac{t^2}{\sigma^2}\right)^{2\mu-1}} dt \notag\\ &=& \frac{2(2\mu-1)}{\gamma^{\mathrm{mod}}\left(\frac{1}{2},\frac{1}{2\sigma^{4\mu-2}}\right)}
 \int_{0}^{1} t^{2k+2\mu-2}
e^{-\left(\frac{t^{2\mu-1}}{\sqrt{2}\sigma^{2\mu-1}}\right)^{2}}dt \notag\\ \notag &=&
    \frac{2}{\gamma^{\mathrm{mod}}\left(\frac{1}{2},\frac{1}{2\sigma^{4\mu-2}}\right)}
    \int_{0}^{1} u^{\frac{2k}{2\mu-1}}   e^{-\left(\frac{u}{\sqrt{2}\sigma^{2\mu-1}}\right)^2}  du. 
  \end{eqnarray}
Applying Lemma~\ref{rho1} with 
\begin{equation*}
    s=\frac{2k}{2\mu-1}, \quad \rho=\frac{1}{\sqrt{2}\sigma^{2\mu-1}}
\end{equation*}
we can write
\begin{equation*}
    \frac{2}{\gamma^{\mathrm{mod}}\left(\frac{1}{2},\frac{1}{2\sigma^{4\mu-2}}\right)}
    \int_{0}^{1} u^{\frac{2k}{2\mu-1}}   e^{-\left(\frac{u}{\sqrt{2}\sigma^{2\mu-1}}\right)^2}  du= \frac{\gamma^{\mathrm{mod}}\left(\frac{2k+2\mu-1}{2(2\mu-1)},\frac{1}{2\sigma^{4\mu-2}}\right)}{\gamma^{\mathrm{mod}}\left(\frac{1}{2},\frac{1}{2\sigma^{4\mu-2}}\right)}.
\end{equation*}
Therefore
\begin{equation*}
    s_{2k,\sigma,\mu}
=\frac{\gamma^{\mathrm{mod}}\left(\frac{2k+2\mu-1}{2(2\mu-1)},\frac{1}{2\sigma^{4\mu-2}}\right)}{\gamma^{\mathrm{mod}}\left(\frac{1}{2},\frac{1}{2\sigma^{4\mu-2}}\right)}.
\end{equation*}
\end{proof}

As an immediate application of the integral representation of the moments in Lemma~\ref{momw1},
we obtain an alternative expression for the polynomial in~\eqref{LegPol2}.

\begin{remark}
The polynomial ${q}_{\sigma,\mu}$ defined in~\eqref{LegPol2} 
can be expressed in terms of the second moment as
\begin{equation}\label{was22}
   q_{\sigma,\mu}(t) = t^2 - s_{2,\sigma,\mu}.
\end{equation}
Indeed, by taking $m=2$ in identity~\eqref{dsssaasasa}, one has
\begin{equation*}
       s_{2,\sigma,\mu} 
   = \frac{\gamma^{\mathrm{mod}}\left(\frac{2\mu+1}{2(2\mu-1)},\frac{1}{2\sigma^{4\mu-2}}\right)}
          {\gamma^{\mathrm{mod}}\left(\frac{1}{2},\frac{1}{2\sigma^{4\mu-2}}\right)},
\end{equation*}
which directly yields~\eqref{was22}.
\end{remark}

 \begin{remark}
Setting $m=0$ in~\eqref{dsssaasasa} gives $s_{0,\sigma,\mu}=1$, confirming that $\widetilde{g}_{\sigma,\mu}$ is a probability density.
\end{remark}

\begin{remark}
We observe that
\begin{equation*}
   \lim_{\sigma\to \infty}\int_{-1}^1 t^m \widetilde{g}_{\sigma,\mu}(t) dt= \int_{-1}^1 t^m \widetilde{g}_{\infty,\mu}(t) dt.
\end{equation*}
Indeed, from the definition~\eqref{w1inf2} we have
\begin{equation*}
 \int_{-1}^{1} t^m \widetilde{g}_{\infty,\mu}(t) dt= \frac{2\mu-1}{2}\int_{-1}^1  t^{m} \left(t^2\right)^{\mu-1}dt. 
\end{equation*}
If $m$ is odd, this integral vanishes by symmetry. If instead $m=2k$, a straightforward calculation yields
\begin{equation*}
     \int_{-1}^{1} t^{2k} \widetilde{g}_{\infty,\mu}(t) dt= \frac{2\mu-1}{2}\int_{-1}^1 t^{2k}\left(t^2\right)^{\mu-1}dt=\frac{2\mu-1}{2k+2\mu-1}. 
\end{equation*}
On the other hand, using~\eqref{dsssaasasa} together with~\eqref{modlimgamma}, we find
\begin{equation*}
s_{2k,\infty,\mu}:=\lim_{\sigma\to\infty} s_{2k,\sigma,\mu} = \frac{2\mu-1}{2k+2\mu-1},
\end{equation*}
which are exactly the moments of order $2k$ associated with the pdfs $\widetilde{g}_{\infty,\mu}(t)$.
\end{remark}

The next result establishes that $q_{\sigma,\mu}$, defined in~\eqref{LegPol2},
is the degree-two orthogonal polynomial associated with the weight $\widetilde{g}_{\sigma,\mu}$ on $[-1,1]$.

\begin{lemma}\label{polp} 
The polynomial $q_{\sigma,\mu}(t)$ is orthogonal to all linear polynomials on $[-1, 1]$ w.r.t. the weight function $\widetilde{g}_{\sigma,\mu}$. 
\end{lemma}
\begin{proof}
Since $q_{\sigma,\mu}(t)\widetilde{g}_{\sigma,\mu}(t)$ is an even function, it suffices to show that
\begin{equation}\label{ortp0}
\int_{-1}^1 q_{\sigma,\mu}(t)\widetilde{g}_{\sigma,\mu}(t)dt = 0.
\end{equation}
Using the expression of $q_{\sigma,\mu}$ from~\eqref{was22}, we have
\begin{eqnarray}
  \int_{-1}^1 q_{\sigma,\mu}(t) \widetilde{g}_{\sigma,\mu}(t) dt &=& \int_{-1}^1 \left( t^2 - s_{2,\sigma,\mu}\right)  \widetilde{g}_{\sigma,\mu}(t)  dt \\
    &=&   \int_{-1}^1 t^2 \widetilde{g}_{\sigma,\mu}(t) dt - s_{2,\sigma,\mu} \int_{-1}^1 \widetilde{g}_{\sigma,\mu}(t)dt\\
    &=&0.\label{finals} 
\end{eqnarray}
In the last equality we used the facts that
\begin{equation*}
    \int_{-1}^1 \widetilde{g}_{\sigma,\mu}(t)dt=1
\end{equation*}
since $\widetilde{g}_{\sigma,\mu}(t)$ is a probability density function, and
\begin{equation*}
    \int_{-1}^1 t^2 \widetilde{g}_{\sigma,\mu}(t) dt=s_{2,\sigma,\mu}
\end{equation*}
which is precisely the second moment of $\widetilde{g}_{\sigma,\mu}(t)$.
\end{proof}

In what follows, we prove several crucial lemmas needed in the sequel.

\begin{lemma}\label{I1}
  For any $i,j=1,2,3$, it holds  
\begin{equation}\label{ex32s}
{\mathcal{D}}_{j}^{\mathrm{enr}}\left(\lambda_i\right)=\frac{1}{2} \left(1-\delta_{ij}\right).
\end{equation} 
\end{lemma} 
\begin{proof}
    The proof is similar to that of Lemma~\ref{I12} and is therefore omitted.
\end{proof}

\begin{lemma}\label{I2}
  For any $i,j=1,2,3$, it holds 
\begin{equation}\label{ex32as}
{\mathcal{D}}_{j}^{\mathrm{enr}}\left(\lambda_i^2\right)=\frac{1+s_{2,\sigma,\mu} }{4}
 \left(1-\delta_{ij}\right),
\end{equation} 
where $s_{2,\sigma,\mu}$ is the second moment of  $\widetilde{g}_{\sigma,\mu}$.
\end{lemma} 
\begin{proof}
     The proof is similar to that of Lemma~\ref{I22} and is therefore omitted.
\end{proof}

\begin{lemma}\label{l1}
 For any $i,j=1,2,3$, it holds 
\begin{equation}\label{ex3211}
{\mathcal{S}}_{{j}}^{\mathrm{enr}}\left(\lambda_i\right)=0.
\end{equation}
\end{lemma}
\begin{proof}
    The result follows directly from~\eqref{ex3} and Lemma~\ref{polp}.
\end{proof}

\begin{remark}
From the definition of $\mathcal{S}_{j}^{\mathrm{enr}}$ and the 
orthogonality of the polynomial $q_{\sigma,\mu}$ established in 
Lemma~\ref{polp}, it follows that
\begin{equation}\label{ex321102ts}
\mathcal{S}_{j}^{\mathrm{enr}}\left(p_1\right)=0, 
\quad j=1,2,3,
\end{equation}
for every linear polynomial $p_1 \in \mathbb{P}_{1}(T)$.
\end{remark}

\begin{lemma}\label{l2}
  For any $i,j=1,2,3$, it holds
\begin{equation}\label{ex321}
{\mathcal{S}}_{{j}}^{\mathrm{enr}}\left(\lambda_i^2\right)= 
\frac{\left\|q_{\sigma,\mu}\right\|^2_{2,{{\sigma,\mu}}}}{4}\left(1-\delta_{ij}\right).
\end{equation} 
\end{lemma} 
\begin{proof}
        The proof is similar to that of Lemma~\ref{F22} and is therefore omitted. 
\end{proof}

Now, we are ready to prove that the triple $\mathcal{H}_{2,\sigma,\mu}^{\mathrm{enr}}$ is unisolvent. \\

\noindent
\textbf{Proof of Theorem~\ref{th1nnewfin112}.}
Using Lemmas~\ref{I1}, \ref{I2}, \ref{l1}, and \ref{l2}, the proof follows along the same lines as that of Theorem~\ref{th1nnewfin11} and is therefore omitted. 
\qed    

\medskip

\noindent
\textbf{Proof of Theorem~\ref{th2allalf2}}\\
Using Lemmas~\ref{I1}, \ref{I2}, \ref{l1}, and \ref{l2}, the proof follows along the same lines as that of Theorem~\ref{th2allalf3r} and is therefore omitted. 

\qed

\subsection{Parameter selection by global mesh-level tuning}

We discuss a complementary strategy to select the parameters of the edge densities 
(e.g., $\mu,\sigma$ for the families introduced in Section~\ref{sec1}). Given a validation set of test functions $$\left\{f^{(r)}\right\}_{r=1}^R$$ and a sequence of meshes $$\left\{\mathcal T_n\right\}_{n=0}^N,$$
we determine the optimal parameters by minimizing the global $L^1$ error over a candidate grid 
$\mathcal P$ of parameter pairs:
\[
\left(\mu^\star,\sigma^\star\right) \in 
\arg\min_{(\mu,\sigma)\in\mathcal P}
\sum_{r}\sum_{n} 
\left\| f^{(r)} - {\pi}_{1,\sigma,\mu}^{{\mathrm{enr}}}\left[f^{(r)}\right]\right\|_{L^1(\Omega;\mathcal T_n)},
\]
where ${\pi}_{1,\sigma,\mu}^{{\mathrm{enr}}}$ is defined in~\eqref{pi1}.
This procedure provides a robust, once-for-all choice of the parameters.

\begin{algorithm}[H]
\caption{Global parameter tuning by grid search}
\label{alg:grid}
\begin{algorithmic}[1]
\Require Validation functions $\left\{f^{(r)}\right\}_{r=1}^R$; 
         meshes $\{\mathcal T_n\}_{n=0}^N$; 
         candidate grid $\mathcal P$ of parameter pairs $(\mu,\sigma)$; 
         reconstruction operator $\pi^{\mathrm{enr}}_{1,\mu,\sigma}$
\Ensure Optimal parameters $\left(\mu^\star,\sigma^\star\right)$
\State $E_{\min} \gets +\infty$; \quad $\left(\mu^\star,\sigma^\star\right)\gets \text{undefined}$
\ForAll{$(\mu,\sigma)\in\mathcal P$}
  \State $E(\mu,\sigma)\gets 0$
  \For{$r=1$ to $R$}
    \For{$n=0$ to $N$}
      \State $u \gets \pi^{\mathrm{enr}}_{1,\mu,\sigma}[f^{(r)};\mathcal T_n]$
      \State $E(\mu,\sigma) \gets E(\mu,\sigma) + \|f^{(r)}-u\|_{L^1(\Omega;\mathcal T_n)}$
    \EndFor
  \EndFor
  \If{$E(\mu,\sigma) < E_{\min}$}
     \State $E_{\min}\gets E(\mu,\sigma)$; \quad $\left(\mu^\star,\sigma^\star\right)\gets(\mu,\sigma)$
  \EndIf
\EndFor
\State \Return $\left(\mu^\star,\sigma^\star\right)$
\end{algorithmic}
\label{alg1}
\end{algorithm}
This grid--search strategy is simple to implement and guarantees the existence of a well--defined 
pair of optimal parameters $\left(\mu^\star,\sigma^\star\right)$. Once selected, these values can be applied 
uniformly across all mesh elements, thus ensuring that the reconstruction procedure is 
completely specified and reproducible. Although the search involves multiple reconstructions on 
the validation set, it is performed only once as an offline preprocessing step. The modest 
computational cost is largely compensated by a substantial gain in robustness and accuracy in 
the proposed weighted histopolation framework. In practice, the method behaves like a standard 
hyperparameter tuning stage in machine learning, providing an automatic and principled way to 
calibrate the edge densities before applying the reconstruction algorithm to new data.


\section{Generalization to Arbitrary Edge Probability Densities}
\label{sec3}
In the previous sections, we focused on two specific two-parameter families of 
generalized truncated normal distributions, each giving rise to valid probability 
densities on the edges of a triangular element. A key observation is that, regardless 
of the chosen family, the resulting edge-restricted densities share the same structural 
properties: they are normalized probability densities, symmetric with respect to the 
edge parameter, and admit finite moments of all orders. These properties alone are 
sufficient to establish unisolvence of the enriched triple and to construct explicit 
quadratic reconstruction operators.

This motivates us to consider a more general setting. Let $T \subset \mathbb{R}^2$ 
be a nondegenerate triangle with edges $s_j = \left[\B v_{j+1}, \B v_{j+2}\right]$, 
and denote by
\begin{equation*}
\gamma_j(t) = \frac{1+t}{2} \B v_{j+1} + \frac{1-t}{2} \B v_{j+2}, 
\quad t \in [-1,1],    
\end{equation*}
the affine parametrization of the edge $s_j$. 
Let $\omega \in L^1([-1,1])$ be a general pdf on $[-1,1]$. We assume further that $\omega$ admits finite moments
\begin{equation}\label{notazz}
\mu_{m,\omega} = \int_{-1}^1 t^m \omega(t)dt, \quad m\in\mathbb{N}.    
\end{equation}
\subsection{Edge functionals and orthogonal polynomials}
Given such a density $\omega$, we define enriched edge functionals in analogy with the 
previous constructions by
\begin{align}
\mathcal{I}^\omega_j(f) &= \int_{-1}^1 f(\gamma_j(t))  \omega(t) dt, \\
\mathcal{L}^\omega_j(f) &= \int_{-1}^1 q(t) f(\gamma_j(t)) \omega(t) dt, \label{ssasaaz}
\end{align}
where $q \in \mathbb{P}_r([-1,1])$, $r \ge 2$, is a polynomial satisfying the orthogonality 
conditions
\begin{equation}\label{eq:orth-min}
    \int_{-1}^1 q(t) \omega(t)dt = 0, 
\quad 
\int_{-1}^1 t q(t) \omega(t)dt = 0,
\end{equation}
together with the non-degeneracy condition on the second moment
\begin{equation}\label{eq:mu2}
    \mu_{2,\omega} := \int_{-1}^1 t^2 q(t) \omega(t)dt \neq 0.
\end{equation}
These assumptions ensure that the $q$-weighted edge functionals~\eqref{ssasaaz} vanish on affine 
polynomials while detecting quadratic contributions.

\subsection{Unisolvency and reconstruction}
We then define the enriched triple
\begin{equation*}
  \mathcal{H}^{\omega} = \left(T, \mathbb{P}_2(T), \Sigma^{\omega}(T)\right),  
\end{equation*}
where
\begin{equation*}
\Sigma^{\omega}(T) = \left\{ \mathcal{I}^{\omega}_j, \mathcal{L}^{\omega}_j \,:\, j=1,2,3 \right\}.    
\end{equation*}
The following result shows that the structural properties observed for the specific 
families extend to any choice of probability density $\omega$.

\begin{theorem}\label{th1nnewfin11ss} 
 The triple $\mathcal{H}^{\omega}$ is unisolvent.
\end{theorem}

\begin{proof}[Sketch of proof]
The argument follows the same lines as in the proofs of Theorems~2.3 and~2.19. Indeed, let $p_2\in \mathbb{P}_2(T)$ such that
\begin{equation*}
    \mathcal{I}_j^{\omega}\left(p_2\right)=\mathcal{L}_j^{\omega}\left(p_2\right)=0, \quad j=1,2,3.
\end{equation*}
Along each edge $s_j$, the trace of $p_2$ is quadratic. By construction, $\mathcal{L}^{\omega}_j$ 
annihilates constant and linear terms, while detecting the quadratic coefficient. 
Hence, vanishing of $\mathcal{L}^{\omega}_j\left(p_2\right)$ implies that all edge traces of $p_2$ are affine. This forces $p_2 \in \mathbb{P}_1(T)$. The remaining conditions $\mathcal{I}^{\omega}_j\left(p_2\right)=0$ then yield a 
homogeneous system for the linear coefficients, which admits only the trivial solution. This completes the proof.
\end{proof}
\medskip

This generalization shows that the enriched quadratic reconstruction framework is not 
restricted to the two-parameter families considered earlier, but in fact applies to 
any choice of probability density on the edges. The specific distributions $\widetilde{k}_{\sigma,\mu}$ 
and $\widetilde{g}_{\sigma,\mu}$ therefore appear as special cases within a broad and flexible class, 
highlighting the robustness and universality of the proposed approach.

\begin{theorem}\label{th2allalf} 
If the probability density function $\omega$ is even, then the basis functions 
associated with the enriched triple $\mathcal{H}^{\omega}$ take the form
\begin{align*}
   \varphi_{\omega, i}&=  1-2\lambda_i,  \quad i=1,2,3, \\
    \psi_{\omega, i}&=  - A_{\omega}\varphi_{\omega, i}+\frac{2}{\left\|{q}\right\|^2_{2,{\omega} }}\left(-\lambda_i^2+\lambda^2_{i+1}+\lambda^2_{i+2}\right), \quad i=1,2,3,
\end{align*}
where 
\begin{equation*}
    \left\|{q}\right\|^2_{2,{\omega} }=\int_{-1}^{1}q^2(t)
\omega(t)dt,
\end{equation*}
\begin{equation*}
A_{\omega}=\frac{1+\mu_{2,\omega}}{ \left\|{q}\right\|^2_{2,{\omega}}}.
\end{equation*}
      \end{theorem}

\subsection{Practical constructions of $q$}
When $q$ is chosen to be a quadratic polynomial, it is possible to construct it explicitly
from the moments $\mu_{m,\omega}$. We outline below three practical approaches.
\begin{enumerate}
  \item[\textbf{(A)}] \emph{Canonical quadratic form.}  Seek $q(t)=t^2-a-bt$ subject to the orthogonality conditions~\eqref{eq:orth-min}.  
Using the notation~\eqref{notazz},  
the orthogonality conditions yield the linear system
  \begin{equation*}
       \begin{cases}
  \mu_{2,\omega}-a-b \mu_{1,\omega}=0,\\
   \mu_{3,\omega}-a \mu_{1,\omega}-b \mu_{2,\omega}=0,
  \end{cases}
  \quad\Longrightarrow\quad
  b=\frac{ \mu_{1,\omega} \mu_{2,\omega}- \mu_{3,\omega}}{ \mu_{1,\omega}^2- \mu_{2,\omega}},\quad a= \mu_{2,\omega}-b \mu_{1,\omega},
  \end{equation*}
  provided $ \mu_{2,\omega}\ne  \mu_{1,\omega}^2$ (nonzero variance). Optionally normalize $q$ by 
  \begin{equation*}
  \widehat q=\frac{q}{\nu}    
  \end{equation*}
   where
   \begin{equation*}
       \nu=\mu_{4,\omega}-a \mu_{2,\omega}-b \mu_{3,\omega}.
   \end{equation*}
  \item[\textbf{(B)}] \emph{Orthogonal polynomial.}
 Construct the family of orthogonal polynomials $\left\{\pi_k\right\}_{k\ge 0}$ with respect to the inner product
\begin{equation*}
    \left\langle f,g\right\rangle_\omega=\int_{-1}^1 f(t)g(t)\omega(t)dt.
\end{equation*}
  Taking $q=\pi_2$ (or its normalized version) ensures that conditions~\eqref{eq:orth-min}–\eqref{eq:mu2} 
  are satisfied by construction~\cite{milovanovic1997orthogonal, mastroianni2008interpolation}.
  \item[\textbf{(C)}] \emph{Higher degree.}   Any polynomial $q$ of degree $\ge 2$ fulfilling~\eqref{eq:orth-min}–\eqref{eq:mu2} is admissible.  
  The quadratic choice is however preferred, since it requires only low-order moments 
  and ensures stability.
\end{enumerate}

\begin{remark}
If the chosen density $\omega$ is symmetric with respect to the origin, 
then all odd moments vanish, i.e., $$\mu_{2k+1,\omega}=0, \quad k\in\mathbb{N}.$$ 
In this case, the expressions in~(A) simplify considerably: one obtains 
$b=0$ and hence
\begin{equation*}
   q(t) = t^2 - \mu_{2,\omega}.  
\end{equation*}
This situation arises in particular for both the first and the second family of densities 
introduced in Section~\ref{sec1}, which are symmetric and therefore 
lead to especially simple quadratic forms.
\end{remark}

\subsection{Numerical procedure}

Before presenting the numerical results, we summarize the computational workflow 
adopted in our experiments. The procedure consists of two phases:

\begin{itemize}
  \item \emph{Offline parameter tuning}: optimal density parameters $\left(\mu^\star,\sigma^\star\right)$ 
  are selected by Algorithm~\ref{alg1}, applied to a small validation set of functions 
  and a sequence of meshes.
  \item \emph{Reconstruction and error evaluation}: once the optimal parameters are fixed, 
  they are used uniformly across all tests to compare the classical histopolation scheme 
  $\mathcal{CH}$ with the enriched scheme $\mathcal{H}^{\mathrm{enr}}_{1,\sigma^\star,\mu^\star}$.
\end{itemize}

The pseudocode of the complete workflow is reported in Algorithm~\ref{alg:tests}.

\begin{algorithm}[H]
\caption{Numerical testing workflow}
\label{alg:tests}
\begin{algorithmic}[1]
\Require Test functions $\left\{f_r\right\}_{r=1}^{R}$; mesh sequence $\left\{\mathcal{T}_n\right\}_{n=0}^N$; parameter grid $\mathcal P$
\Ensure $L^1$ errors for $\mathcal{CH}$ and $\mathcal{H}^{\mathrm{enr}}_{1,\sigma^\star,\mu^\star}$
\State $\left(\mu^\star,\sigma^\star\right) \gets$ parameter tuning via Algorithm~\ref{alg1}
\For{each test function $f_j$}
  \For{each mesh $\mathcal{T}_n$}
    \State $u^{\mathrm{CH}} \gets \pi^{\mathrm{CH}}\left[f_j;\mathcal{T}_n\right]$
    \State $u^{\mathrm{enr}} \gets \pi_{1,\sigma,\mu}^{\mathrm{enr}}\left[f_j;\mathcal{T}_n\right]$
    \State Compute $E^{\mathrm{CH}}_{j,n} =\left \|f_j - u^{\mathrm{CH}}\right\|_{L^1(\Omega;\mathcal{T}_n)}$
    \State Compute $E^{\mathrm{enr}}_{j,n} = \left\|f_j - u^{\mathrm{enr}}\right\|_{L^1(\Omega;\mathcal{T}_n)}$
  \EndFor
\EndFor
\State \Return Error arrays $\left\{E^{\mathrm{CH}}_{j,n}, E^{\mathrm{enr}}_{j,n}\right\}$
\end{algorithmic}
\end{algorithm}

This explicit workflow highlights how the parameter tuning step integrates seamlessly into the 
testing phase: after a one-time offline calibration, the enriched scheme is applied to new test 
functions and meshes in a fully specified and reproducible way.

\section{Numerical tests}\label{sec4}
In this section, we present a series of numerical experiments to assess the performance 
and accuracy of the proposed enriched histopolation methods. 
All tests are carried out on the square domain $\Omega = [-1,1]^2$, 
using a collection of benchmark functions that exhibit a variety of analytic behaviors, 
including smooth, oscillatory, and localized features. 
The test functions are defined as follows
 \begin{eqnarray*}
    f_1(x,y)&=&\sqrt{x^2+y^2}, \ \ f_2(x,y)=e^{-4\left(x^2+y^2\right)}\sin\left(\pi(x+y)\right),\\
    f_3(x,y)&=&\sin(2\pi x)\sin(2\pi y),\ \, 
    f_4(x,y)=\sin\left(4\pi(x+y)\right),\ \,
    f_5(x,y)=\frac{1}{25(x^2+y^2)+1},\\ 
 f_6(x,y)&=&0.75\exp\biggl(-\frac{(9(x+1)/2-2)^2}{4}-\frac{(9(y+1)/2-2)^2}{4}\biggr)\\[-2pt]
 &&+0.75\exp\biggl(-\frac{(9(x+1)/2+1)^2}{49}-\frac{(9(y+1)/2+1)}{10}\biggr)\\[-2pt]
	&&+ 0.5\exp\biggl(-\frac{(9(x+1)/2-7)^2}{4}-\frac{(9(y+1)/2-3)^2}{4}\biggr)\\
	&&-0.2\exp\biggl(-(9(x+1)/2-4)^2-(9(y+1)/2-7)^2\biggr).
\end{eqnarray*}
The function $f_6$ is the well-known Franke function, widely used as a benchmark for approximation methods~\cite{franke1982scattered}. 
For the spatial discretization, we consider a sequence of regular Friedrichs--Keller 
triangulations~\cite{Knabner}, denoted by
\begin{equation*}
    \mathcal{T}_n=\left\{T_i : i=1,\dots,2(n+1)^2\right\},
\end{equation*}
where each $\mathcal{T}_{n}$ consists of $2(n+1)^2$ triangles; see Fig.~\ref{Fig:rec0}.
\begin{figure}
    \centering
\includegraphics[scale=0.49]{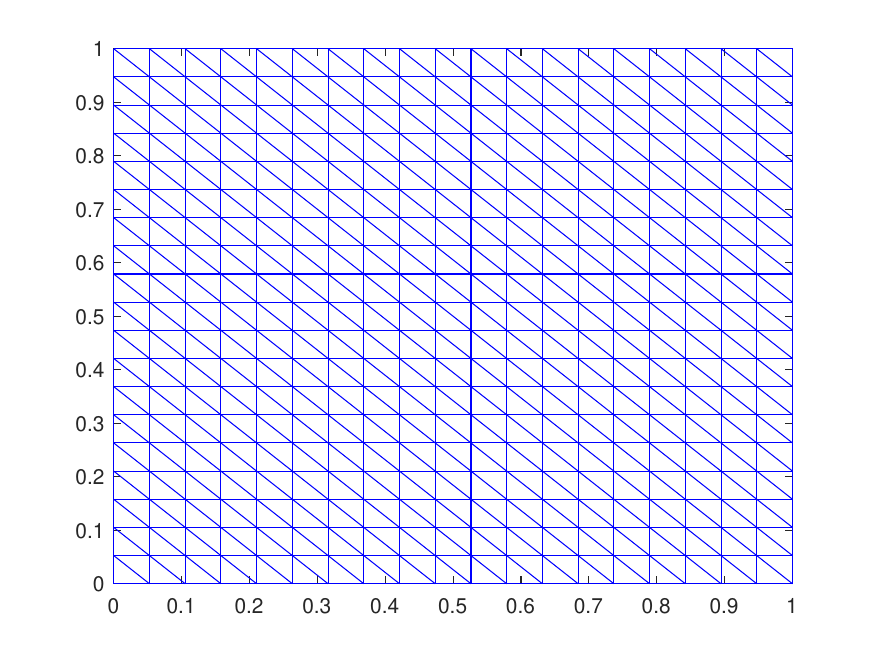}
\includegraphics[scale=0.49]{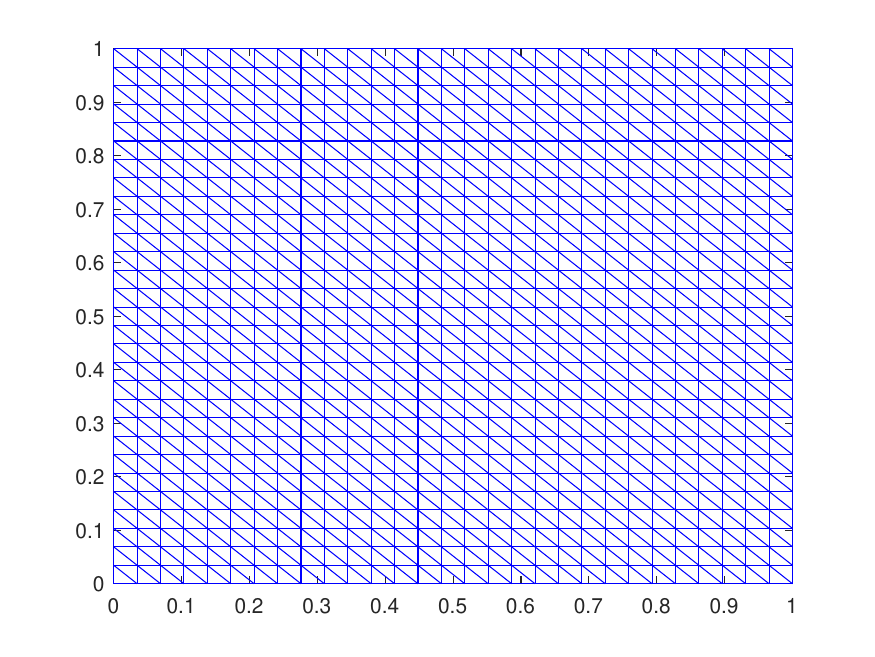}
\includegraphics[scale=0.49]{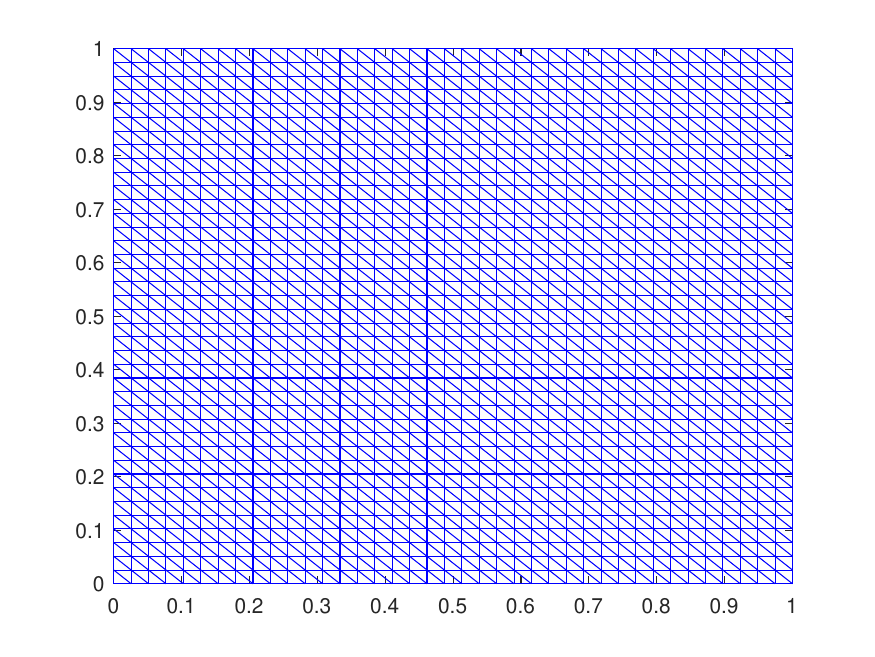}
\includegraphics[scale=0.49]{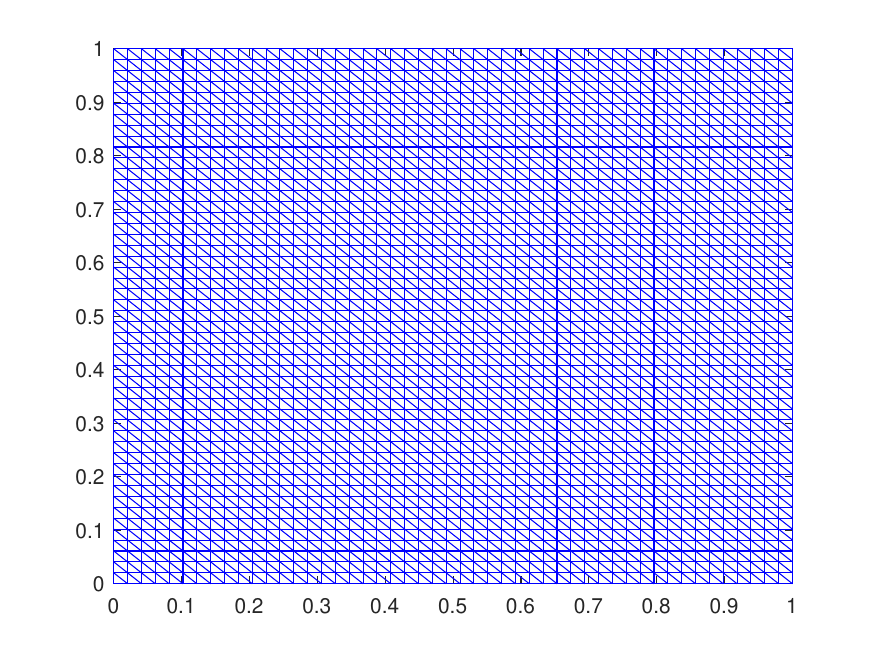}
    \caption{Successive regular Friedrichs--Keller triangulations $\mathcal{T}_n$, with $n=20$ (top left), $n=30$ (top right), $n=40$ (bottom left), and $n=50$ (bottom right).}
    \label{Fig:rec0}
\end{figure}

To evaluate the quality of the reconstruction, we compute the error in the $L^1$ norm 
produced by the classical linear histopolation method $\mathcal{CH}$ defined in~\eqref{CRelement} 
and compare it with the error obtained using the enriched histopolation method 
$\mathcal{H}^{\mathrm{enr}}_{1,\sigma^\star,\mu^{\star}}$ introduced in~\eqref{tripless}, 
with the optimal parameters $\mu^{\star}=2$ and $\sigma^{\star}=1$. 
These parameters have been determined by means of the parameter tuning 
Algorithm~\ref{alg1}. All experiments have been carried out according to Algorithm~\ref{alg:tests}, implemented in \textsc{Matlab}. 
The integrals needed for assembling the reconstruction operators and for evaluating the $L^1$ errors 
were computed on each triangular element using high-order quadrature rules, with a sufficiently 
dense set of nodes to ensure accurate numerical integration and to rule out any artifacts due to 
quadrature errors.  The $L^1$ norm is chosen as it handles discontinuities effectively without requiring limiting procedures (cf.~\cite{dell2025truncated}). 
The results of this comparison are reported in Figures~\ref{im12},~\ref{im34} and~\ref{im56}. 
A clear improvement of the enriched method over the standard scheme is observed across 
all test functions and mesh refinements. In particular, 
$\mathcal{H}^{\mathrm{enr}}_{1,\sigma^{\star},\mu^{\star}}$ achieves significantly smaller $L^1$ errors, 
demonstrating its superior capability in reproducing local oscillations, sharp gradients, 
and singular features. Moreover, the enriched approach demonstrates an accelerated error decay under mesh refinement, thereby confirming that the additional degrees of freedom provided by the probability density weights yield a genuine enhancement in approximation power.

\begin{figure}
  \centering
   \includegraphics[width=0.49\textwidth]{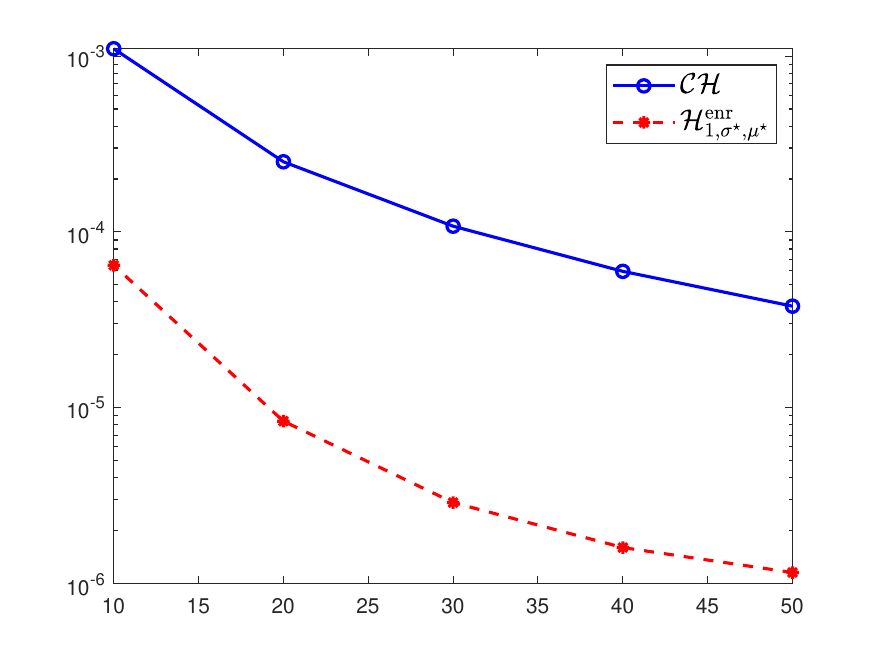} 
\includegraphics[width=0.49\textwidth]{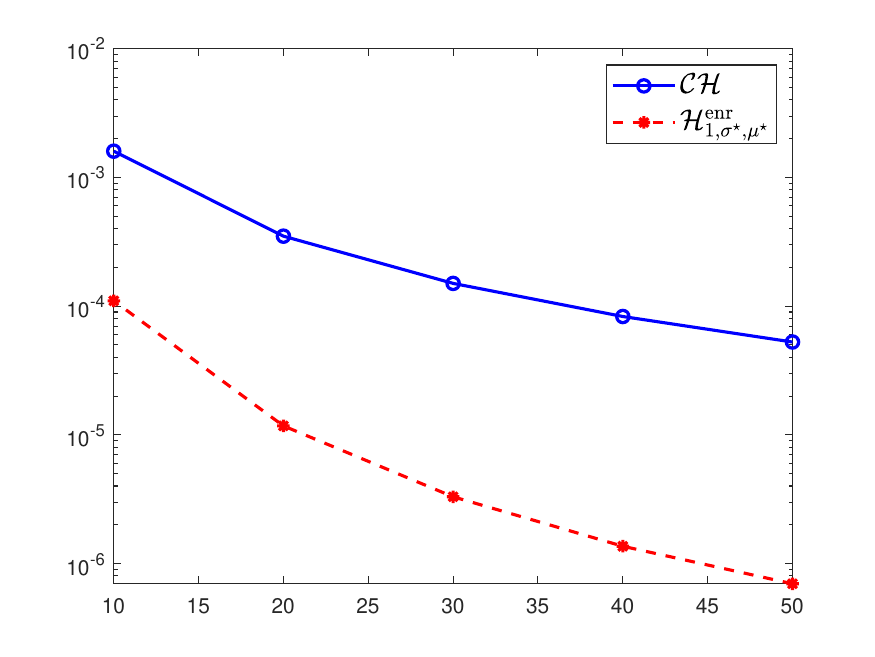} 
         \caption{Semi-log plot of the $L^1$ approximation error for $f_1$ (left) and $f_2$ (right) 
obtained with the classical histopolation method $\mathcal{CH}$ (blue) 
and the weighted quadratic enriched method 
$\mathcal{H}^{\mathrm{enr}}_{1,\sigma^{\star},\mu^{\star}}$ (red), 
as the number of triangles in the Friedrichs--Keller triangulations increases.}
\label{im12}
\end{figure}

\begin{figure}
  \centering
   \includegraphics[width=0.49\textwidth]{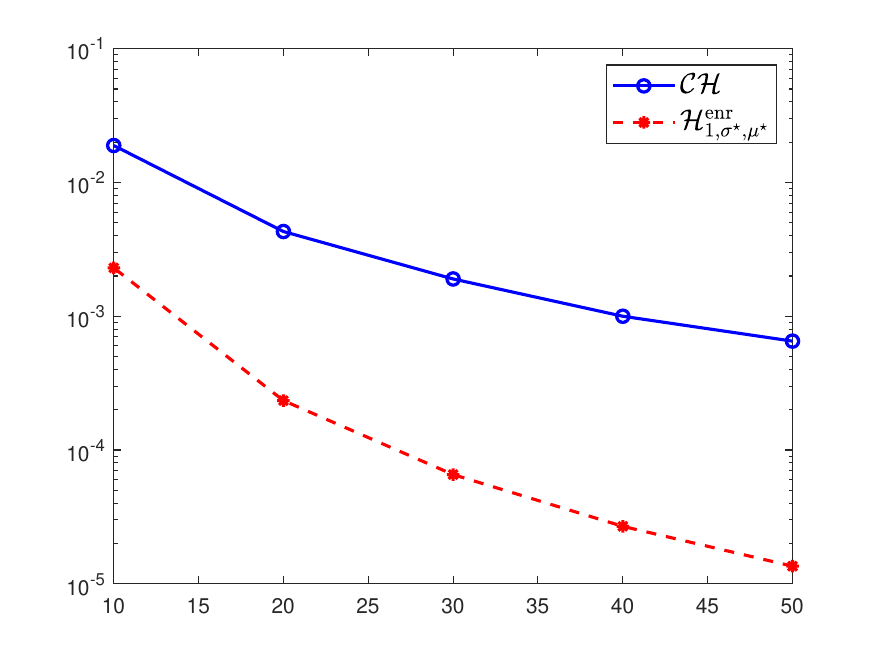} 
\includegraphics[width=0.49\textwidth]{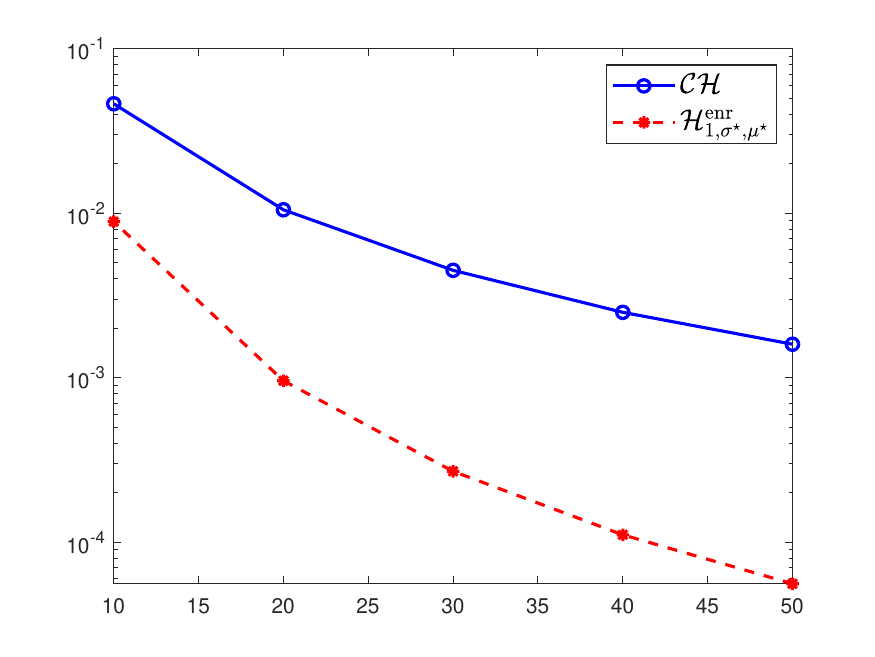} 
         \caption{Semi-log plot of the $L^1$ approximation error for $f_3$ (left) and $f_4$ (right) 
obtained with the classical histopolation method $\mathcal{CH}$ (blue) 
and the weighted quadratic enriched method 
$\mathcal{H}^{\mathrm{enr}}_{1,\sigma^{\star},\mu^{\star}}$ (red), 
as the number of triangles in the Friedrichs--Keller triangulations increases.}
\label{im34}
\end{figure}

\begin{figure}
  \centering
   \includegraphics[width=0.49\textwidth]{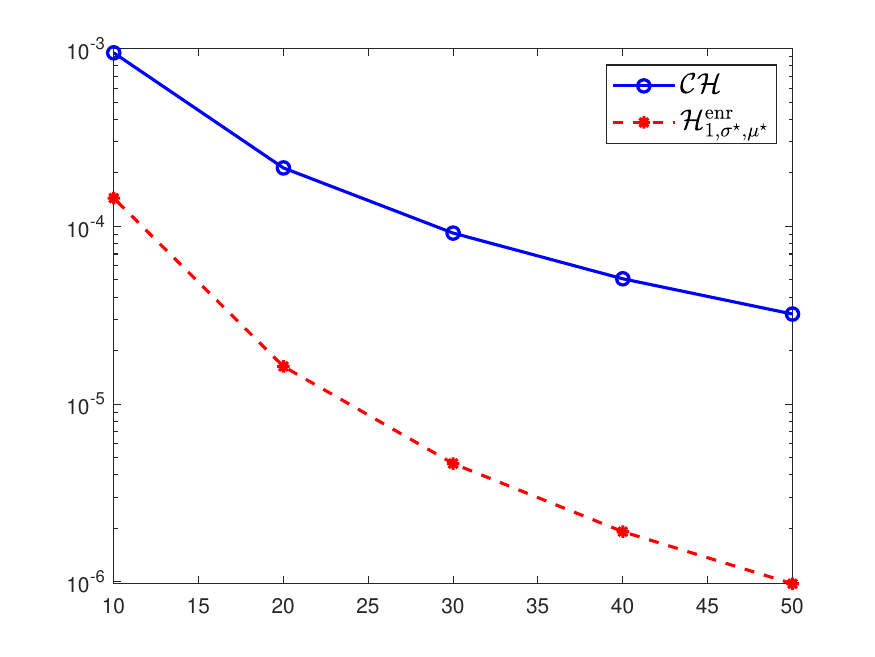} 
\includegraphics[width=0.49\textwidth]{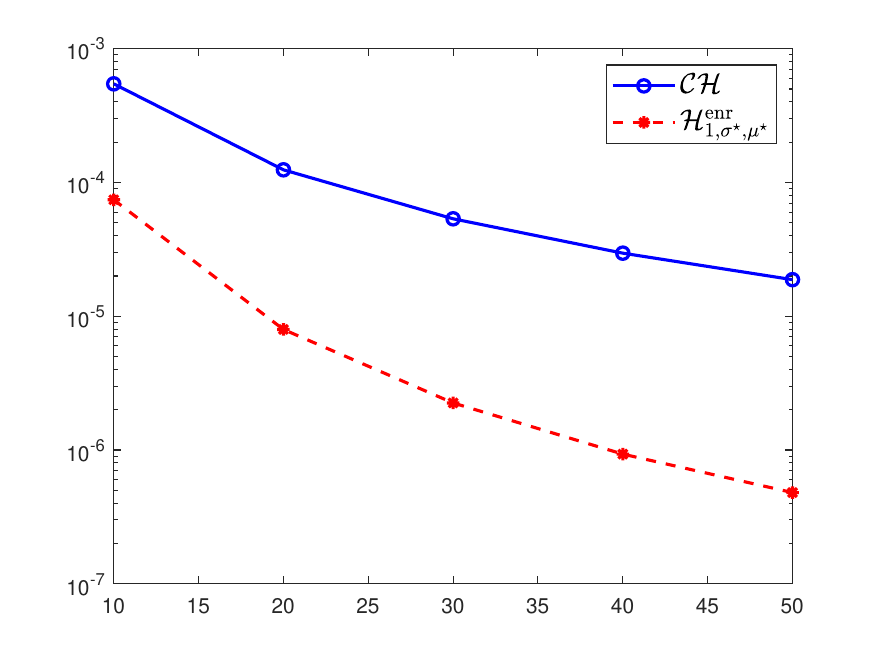} 
         \caption{Semi-log plot of the $L^1$ approximation error for $f_5$ (left) and $f_6$ (right) 
obtained with the classical histopolation method $\mathcal{CH}$ (blue) 
and the weighted quadratic enriched method 
$\mathcal{H}^{\mathrm{enr}}_{1,\sigma^{\star},\mu^{\star}}$ (red), 
as the number of triangles in the Friedrichs--Keller triangulations increases.}
\label{im56}
\end{figure}

\section{Conclusions and Future works}\label{sec5}
In this work we have introduced two new two-parameter families of generalized truncated normal 
distributions and exploited them to design enriched local histopolation schemes based on unisolvent triples. 
These distributions serve as flexible edge weights and provide additional degrees of freedom, which 
translate into improved accuracy and adaptability of the reconstruction operators. The resulting 
quadratic operators preserve the simplicity of the classical linear histopolation method while 
significantly enhancing their approximation capabilities, especially for functions with oscillations, 
steep gradients, or localized singularities. 

From a theoretical point of view, we have established unisolvency and derived explicit closed-form basis functions, thus ensuring that the proposed schemes are mathematically well-posed and computationally viable. We have also proposed an algorithm for the optimal selection of the distribution parameters, which further increases the robustness and adaptivity of the approach. Numerical experiments have confirmed the theoretical findings, showing systematic improvements over the linear nonconforming histopolation approach. Moreover, as shown in Section~\ref{sec3}, the same reasoning extends to any bivariate weight function whose restriction to the edges of the triangle defines a valid probability density. This highlights the generality and robustness of the proposed framework, which is not limited to the two specific families presented in this paper.

Several directions for future research naturally arise from this study. First, extending the 
framework to three-dimensional meshes would broaden its applicability to volumetric data 
reconstruction and tomographic imaging. Second, adaptive strategies for selecting the distribution 
parameters could be developed, so as to optimize accuracy depending on local features of the target 
function. Third, connections with other classes of orthogonal polynomials and probability densities 
could be investigated, possibly leading to new families of enriched operators. Finally, applications 
to real-world problems in imaging, computer vision, and numerical approximation remain a promising 
avenue for further exploration.

\section*{Acknowledgments}
This research has been achieved as part of RITA \textquotedblleft Research
 ITalian network on Approximation'' and as part of the UMI group \enquote{Teoria dell'Approssimazione
 e Applicazioni}. The research was supported by GNCS–INdAM 2025 project \emph{``Polinomi, Splines e Funzioni Kernel: dall'Approssimazione Numerica al Software Open-Source''}. 
The work of F. Nudo is funded from the European Union – NextGenerationEU under the Italian National Recovery and Resilience Plan (PNRR), Mission 4, Component 2, Investment 1.2 \lq\lq Finanziamento di progetti presentati da giovani ricercatori\rq\rq,\ pursuant to MUR Decree No.~47/2025. The research was supported by  by the grant \enquote{Bando Professori visitatori 2022} which has allowed the visit of Prof. Allal Guessab to the Department of Mathematics and Computer Science of the University of Calabria in the spring 2022. The authors extend their appreciation to the Deanship
of Research and Graduate Studies at King Khalid University for funding
this work through Large Research Project under grant number
RGP2/305/46.

\section*{Conflict of interest}
Not Applicable.

\bibliographystyle{spmpsci}
\bibliography{bibliografia}

 \end{document}